\input ./style/arxiv-general.cfg
\documentclass[MSNbibl,number,citesort,seceqn,dvips]{arxbj}
\makeatletter
   \@ifpackageloaded{graphicx}{}{\usepackage{graphicx}}
\makeatother
\usepackage{mathbh}

\volume{22}
\issue{3}
\pubyear{2016}
\firstpage{1448}
\lastpage{1490}
\doi{10.3150/15-BEJ699}
\docsubty{FLA}

\makeatletter
\renewcommand{\mid}{\vert}

\def\vfrac#1#2{(#1)/#2}

\def\vafrac#1#2{(#1)/(#2)}
\def\sklfrac#1#2{(#1/#2)}

\newcommand{\rrvert}{\vert}
\newcommand{\rrVert}{\Vert}
\newcommand{\llvert}{\vert}
\newcommand{\llVert}{\Vert}
\newcommand{\bigtimes}{\mathop{\,\mbox{\parbox[c][9pt][b]{18pt}{\fontsize{18}{18}\selectfont{$\times$}}}\!\!}}
\newcommand{\bigtimesd}{\mathop{\,\mbox{\parbox[c][15pt][b]{27pt}{\fontsize{27}{27}\selectfont{$\times$}}}\!\!}}
\newcommand{\mathds}{\mathbh}
\newtheorem{theorem}{Theorem}[section]
\newtheorem{lemma}[theorem]{Lemma}
\newtheorem{corollary}[theorem]{Corollary}
\newproclaim{definition}[theorem]{Definition}
\newtheorem{prop}[theorem]{Proposition}
\newremark{remark}[theorem]{Remark}
\newremark{example}[theorem]{Example}
\def\eps{{\varepsilon}}
\def\OO{{\mathcal{O}}}
\def\SS{{\mathcal{S}}}
\def\E{{\mathbb{E}}}
\def\N{{\mathbb{N}}}
\def\R{{\mathbb{R}}}
\def\Z{{\mathbb{Z}}}
\makeatother

\begin{document}
\begin{frontmatter}

\title{A stochastic volatility model with flexible extremal dependence structure}
\runtitle{An SV model with flexible extremal dependence structure}

\begin{aug}
\author{\inits{A.}\fnms{Anja}~\snm{Janssen}\corref{}\ead[label=e1,mark]{anja.janssen@math.uni-hamburg.de}\thanksref{e1}}
\and
\author{\inits{H.}\fnms{Holger}~\snm{Drees}\ead[label=e2,mark]{holger.drees@math.uni-hamburg.de}\thanksref{e2}}
\address{Department of Mathematics, University of Hamburg,
Bundesstr. 55, 20146 Hamburg, Germany.\\ \printead{e1,e2}.}
\end{aug}

%
\received{\smonth{10} \syear{2013}}
%
\revised{\smonth{8} \syear{2014}}

%
\begin{abstract}
Stochastic volatility processes with heavy-tailed innovations are a
well-known model
for financial time series. In these models, the extremes of the log
returns are mainly
driven by the extremes of the i.i.d. innovation sequence which leads
to a very strong
form of asymptotic independence, that is, the coefficient of tail
dependence is equal
to $1/2$ for all positive lags. We propose an alternative class of
stochastic volatility
models with heavy-tailed volatilities and examine their extreme value behavior.
In particular, it is shown that, while lagged extreme observations are typically
asymptotically independent, their coefficient of tail dependence can
take on
any value between $1/2$ (corresponding to exact independence) and 1
(related to asymptotic dependence).
Hence, this class allows for a much more flexible extremal dependence between
consecutive observations than classical SV models and can thus describe the
observed clustering of financial returns
more realistically.

The extremal dependence structure of lagged observations is analyzed in
the framework of regular
variation on the cone $(0,\infty)^d$. As two auxiliary results which
are of interest on
their own we derive a new Breiman-type theorem about regular variation
on $(0,\infty)^d$
for products of a random matrix and a regularly varying random vector
and a statement
about the joint extremal behavior of products of i.i.d. regularly
varying random variables.
\end{abstract}

%
\end{frontmatter}

\section{Introduction}\label{sec1}
\subsection{Extremal behavior of models for financial time series}
Univariate time series of (log-)returns are usually described by
multiplicative models of the form
%
\begin{equation}
\label{eqmultimodel} X_t=\sigma_t\eps_t,\qquad t
\in\Z,
\end{equation}
where $\eps_t$, $t\in\Z$, are i.i.d. innovations and $(\sigma
_t)_{t\in\Z}$ is a non-negative stationary time series of so-called
volatilities. The two most popular classes of multiplicative models
vary in the way the volatilities are modeled. While $\sigma_t$ is a
function of past innovations $\eps_s$, $s<t$, in GARCH-type models,
stochastic volatility models (SV models, for short) assume in contrast
that the volatilities are driven by a second time series $(\eta_t)_{t
\in\mathbb{Z}}$ of innovations. More precisely, it is often assumed
that the log-volatilities are described by a Gaussian linear time
series of the type
%
\begin{equation}
\label{eqlogvol} \log\sigma_t =\sum_{i=0}^\infty
\alpha_i\eta_{t-i},\qquad t\in\Z,
\end{equation}
with i.i.d. normal innovations $\eta_t$, $t\in\Z$, independent of
$(\eps_t)_{t\in\Z}$ (although, sometimes, $\eps_t$ and $\eta
_{t+1}$ are assumed to be correlated to capture a leverage effect).
Because returns are usually heavy-tailed and the volatilities are
lognormal, in this modeling approach the innovations $\eps_t$ are
often assumed to be regularly varying, that is,
%
\begin{equation}
\label{eqreginno} \frac{P(\eps_t>sx)}{P(\llvert \eps_t\rrvert >x)}\to
ps^{-\alpha}, \qquad\frac
{P(\eps_t<-sx)}{P(\llvert \eps_t\rrvert >x)}
\to(1-p)s^{-\alpha}, \qquad x\to\infty,
\end{equation}
for all $s>0$, some $\alpha>0$ and $p\in[0,1]$. By Breiman's lemma
[see Breiman \cite{Br65}] this implies that $X_t$ is
regularly varying as well.

While the (univariate) tails of $X_t$ behave similarly in these SV and
GARCH-type models, the extreme value dependence of consecutive returns
differs between the two model classes. This can firstly be seen from
the conditional probabilities of an extreme event at lag $h$, given an
extreme event at time 0. For GARCH-type models,
%
\begin{equation}
\label{asymptdep}\lim_{x \to\infty
}P(X_h>x\mid
X_0>x)>0,\qquad h \neq0,
\end{equation}
while for SV models as the one above
%
\begin{equation}
\label{asymptindep}\lim_{x \to\infty
}P(X_h>x\mid
X_0>x)=0,\qquad h \neq0
\end{equation}
[cf. Basrak, Davis and Mikosch \cite{BaDaMi02} and Davis and Mikosch
\cite{DaMi09}]. We call the vector $(X_0, X_h)$
asymptotically dependent in the first case and asymptotically
independent in the second case.

The differences between (\ref{asymptdep}) and (\ref{asymptindep}) are
mirrored by a different {\em asymptotic} cluster behavior of the
related models. In case of asymptotic dependence exceedances over
extreme thresholds typically occur in clusters, while in time series
with asymptotically independent consecutive observations exceedances
tend to be separated in time as the threshold increases. It is
important to note, though, that for realistic sample sizes it will
often be difficult to spot the difference with the naked eye, because
sufficiently high thresholds are not exceeded. Figure~\ref{figure1} displays
realizations of an SV time series [with AR(1) log-volatilities used in
(\ref{eqlogvol})] on the left-hand side and of a GARCH$(1,1)$ time
series on the right-hand side. The parameters were obtained by fitting
these models to a stretch of returns from the S\&P 500 stock index.
Although the left plot shows a realization from a model with
asymptotically independent consecutive observations, it nevertheless
exhibits a
strong clustering of extremes, because the asymptotic
independence becomes obvious only for exceedances over higher thresholds.
Hence, the well-documented fact that large losses (and gains) often
occur in clusters does not rule out SV models for the description of
the extreme value dependence between consecutive returns. Indeed, the
analysis of the extreme value dependence between consecutive returns
from UBS stocks and Google stocks in Drees {\it et~al.} \cite
{Dretal14} indicates that, despite the obvious clusters of large losses
and gains, it is appropriate to assume asymptotic independence between
observations of the same sign.

%
\begin{figure}

\includegraphics{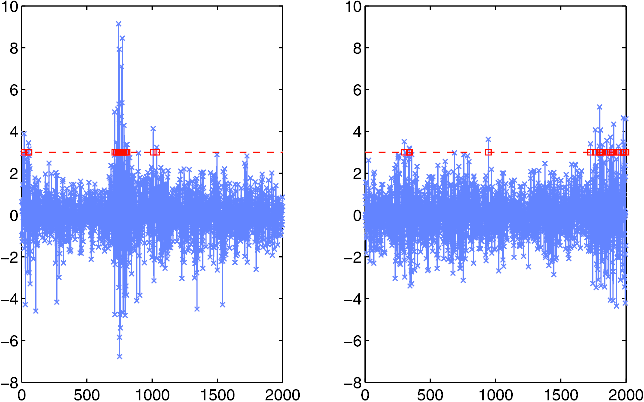}

\caption{Large exceedances for 2000 consecutive observations of a
simulated SV model with heavy-tailed (Student $t$) innovations (left)
and a GARCH($1,1$) process with normal innovations (right). The model
parameters for both models are estimated from the same data set,
consisting of 2432 daily log returns (given in percent) of the S\&P 500
stock index from 1999 to 2008, cf. Abanto-Valle {\it
et~al.} \cite
{AbBaLaEn10} for the SV parameters. The threshold 3 is approximately
equal to the 99\% quantile of the marginal distribution in both models.}
\label{figure1}
\end{figure}


In case of asymptotic independence, the strength of clustering of
exceedances over high, but not extreme thresholds will depend on the
rate at which
\[
P(X_h>x\mid X_0>x)=\frac{P(X_0>x, X_h>x)}{P(X_0>x)},\qquad h \neq0,
\]
tend to 0.
In the present paper, we thus focus on the asymptotics for the
probabilities of joint exceedances $P(X_h>x,X_0>x), h \neq0$, as $x
\to\infty$. Note that established models for financial time series
are quite limited in their ability to reflect different behaviors of
these probabilities. In particular, GARCH models imply that
\[
\lim_{x \to\infty}\frac{P(X_0>x, X_h>x)}{P(X_0>x)} \in(0, \infty),\qquad h
\neq0,
\]
while classical SV models lead to
\[
\lim_{x \to\infty}\frac{P(X_0>x, X_h>x)}{(P(X_0>x))^2} \in(0, \infty),\qquad
h \neq0.
\]
However, applying the methodology developed by Draisma {\it et~al.}
\cite{Draetal04} to consecutive observations in time series of
returns shows that often $P(X_0>x, X_h>x)$ is apparently of the order
$(P(X_0>x))^{1/\eta_h}$ for some $\eta_h$ strictly between $1/2$ and 1,
indicating {\em asymptotic} independence,\vadjust{\goodbreak} but a stronger dependence
between exceedances over high thresholds than implied by classical SV models.
We will therefore introduce an SV model which allows for a much more
flexible modeling of the asymptotics of joint exceedance probabilities.

\subsection{Regular variation on cones}
In order to describe the extremal dependence structure of a time
series, we will use the concept of regular variation of random vectors.
An $\R^d$-valued random vector $\mathbf{Z}$ is said to be regularly
varying on $[0,\infty)^d\setminus\{\mathbf{0\}}$ if there exists a
measure $\tilde\nu\ne0$ on $\mathcal{B}([0,\infty)^d\setminus\{
\mathbf{0\}})$ which is finite on
$[0,\infty)^d\setminus[0,x]^d$ for all $x>0$ such that
%
\begin{equation}
\label{eqmultiregvar} \frac{P(\mathbf{Z}\in xB)}{P(\llVert \mathbf
{Z}\rrVert >x)} \to\tilde\nu(B),\qquad x\to\infty,
\end{equation}
for all Borel sets $B \subset[0,\infty)^d$ with $\tilde\nu(\partial
B)=0$ that are bounded away from the origin $\mathbf{0}$. (Here
$\partial B$ denotes the topological boundary of $B$ and $\llVert \cdot
\rrVert $
is an arbitrary norm on $\R^d$.) By this notion, we follow the
definition of $\mathbb{M}_0$-convergence of Hult and Lindskog
\cite{HuLi06} which can be shown to be equivalent to vague
convergence on $[0, \infty]^d \setminus\{\mathbf{0}\}$. See, for
example, Resnick \cite{Re07}, Section~6.1, for details
on the definition of multivariate regular variation using vague
convergence. The limit measure $\tilde\nu$ is necessarily homogeneous
of some order $-\tilde{\alpha}<0$, which is called the index of
regular variation.

For GARCH time series, Basrak, Davis and Mikosch \cite{BaDaMi02} show
that under general conditions the vectors $(\sigma_{t_1},\ldots,\sigma
_{t_d})$ and $(X_{t_1},\ldots,X_{t_d})$ are multivariate
regularly varying for all $t_1<\cdots<t_d$, and the same holds true
for SV models with heavy-tailed innovations $\eps_t$. However, while
in the former case the limiting measure $\tilde\nu$ puts mass on
$(0,\infty)^d$, it is concentrated on the axes for SV models (cf. Davis
and Mikosch \cite{DaMi01}). This asymptotic independence
of lagged returns (and volatilities) renders convergence (\ref
{eqmultiregvar}) rather uninformative. In particular, we can merely
conclude that the probability $P(X_0>x,X_h>x)$ of joint exceedances is
of smaller order than $P(X_0>x)$ for all $h \neq0$, but we neither
obtain its rate of convergence to 0, nor whether the probability can be
standardized in a different way than in (\ref{eqmultiregvar}) to
obtain a non-trivial limit. Therefore, in the case of asymptotic
independence there is need for a refined analysis of the second-order
extremal dependence behavior.

This can be most elegantly done in the framework of regular variation
on the cone $\mathbb{E}^d:=(0, \infty)^d$.
In what follows, we use the abbreviation $ \min(\mathbf{z}):=\min\{
z_1, \dots, z_d\}$ for $\mathbf{z}=(z_1, \dots, z_d) \in\R^d$.
Now, instead of (\ref{eqmultiregvar}), we assume that there exists a
measure $\nu\ne0$ on $\mathcal{B}(\E^d)$ which is finite on
$(x,\infty)^d$ for all $x>0$ such that
%
\begin{eqnarray}
\label{vagueconvergence} \frac{P(\mathbf{Z} \in
xB)}{P(\min(\mathbf{Z})>x)} \to\nu(B),\qquad x \to\infty,
\end{eqnarray}
for all Borel sets $B \subset\mathbb{E}^d$ with $\nu(\partial B)=0$
that are bounded away from the topological boundary
\[
\OO^d:= \partial\bigl(\mathbb{E}^d\bigr) = \bigl\{
\mathbf{x} \in[0, \infty)^d \mid\min(\mathbf{x})=0\bigr\} %
\]
of the cone $\mathbb{E}^d$. Notice that here we consider events in
which all components of $\mathbf{Z}$ are large, whereas in (\ref
{eqmultiregvar}) just one coordinate needs to be extreme. Again $\nu$
is homogeneous of some order $-\alpha<0$, which is called the index of
regular variation on $\E^d$. If both (\ref{eqmultiregvar}) and (\ref
{vagueconvergence}) hold, then $\alpha\geq\tilde{\alpha}$. In the
case of asymptotic independence (i.e., $\tilde{\nu}$ is concentrated
on $[0,\infty)^d \setminus(0,\infty)^d$), the random vector $\mathbf
{Z}$ is then said to exhibit hidden regular variation.
See Resnick \cite{Re02,Re08} and \cite{Re07}, Section~9.4,
Das, Mitra and Resnick \cite{DaMiRe11} and Lindskog, Resnick and Roy \cite
{LiReRo13} for further details about hidden regular variation and
regular variation on cones.

Kulik and Soulier \cite{KuSo13a} show that for SV models with
heavy-tailed innovations the vector $(X_{t_1},\ldots,X_{t_d})$ of
lagged returns is regularly varying on $(0,\infty)^d$ and that the
limit measure $\nu$ is the same as if the components of this vector
were exactly independent. A similar result holds for the absolute
returns. In particular, $P(X_0>x,X_h>x)$ is of the same order as
$(P(X_0>x))^2$, which means that the coefficient of tail dependence
introduced in Ledford and Tawn \cite{LeTa96} and Ledford and Tawn \cite
{LeTa03} equals $\eta_h=1/2$ for all lags $h>0$. Recall that a
bivariate random vector $(Y_0,Y_1)$ with equal marginal
quantile function $F^\leftarrow$ has a coefficient of tail dependence
$\eta$ if
$x \mapsto P(Y_0>F^\leftarrow(1-1/x),Y_1>F^\leftarrow(1-1/x))$ is a
regularly varying function with index $-1/\eta$. Thus, for independent
$Y_0, Y_1$ the value of $\eta$ is equal to $1/2$ and the higher the
value of $\eta$, the stronger is the extremal dependence of $Y_0$ and $Y_1$.

Hence, the classical SV time series as described above show a very weak
extremal dependence, which is barely influenced by the parameters of
the model. It is the main aim of the present paper to propose and
analyze a modified class of SV models which allows for a much more
flexible and realistic extremal dependence than the classical version.
\subsection{Outline}
Our following analysis of different multiplicative models heavily
relies on the fact that a product of two independent factors inherits
both its tail behavior and extremal dependence from the factor with the
heavier tail. This general heuristic principle was formalized by
Breiman \cite{Br65} for univariate random variables. A
similar ``Breiman-type'' result on the first order dependence behavior
for a product of a random matrix $\mathbf{A}$ and a random vector
$\mathbf{Z}$ was proved by Basrak, Davis and Mikosch \cite
{BaDaMi02}, who analyzed regular variation on $\R^d\setminus\{\mathbf
{0}\}$ and on $[0,\infty)^d\setminus\{\mathbf{0}\}$. In Section~\ref
{sectBreiman}, we establish an analogous result for regular variation
on $\mathbb{E}^d=(0, \infty)^d$, which is somewhat more involved,
because one has to keep in mind that (\ref{vagueconvergence}) only
describes the asymptotic behavior of $\mathbf{Z}/x$ on sets $B$ that
are bounded away from the boundary $\OO^d$ of $\mathbb{E}^d$, and
this feature is not necessarily preserved under
multiplication with a
general matrix $\mathbf{A}$.

Our aim is to apply this Breiman result to SV models with a
heavy-tailed volatility sequence and light-tailed innovations in order
to show that for these models the second-order behavior of extremes is
mainly determined by the volatility sequence. In Section~\ref
{sectSVmodels}, we introduce a particular class of such models with
Gamma-type log-volatilities and derive their marginal tail behavior
using results from Rootz{\'e}n \cite{Ro86}. Moreover, the
first-order extremal behavior is analyzed in terms of point processes
of exceedances. Similar SV models which allow for both asymptotic
dependence and asymptotic independence have been proposed in
Mikosch and Rezapour \cite{MiRe12}, but their analysis is
restricted to the in this case rather uninformative first-order
extremal dependence of those models, while our focus is on the refined
second-order behavior in the asymptotically independent case.

In our SV models, the volatilities are given as (generally infinite)
products of powers of regularly varying i.i.d. random variables.
Section~\ref{sectpowerproducts} deals with the asymptotic behavior of
joint exceedance probabilities of such products, which turns out to be
intimately related to the solution of certain linear optimization
problems. While the heuristics for this connection can easily be
understood, the exact arguments are more delicate. We give a full
result in the case of two products and the proof of a special case
together with an outlook for the general case of more than two
products. In Section~\ref{sectasymptoticsSV}, we discuss the
consequences for our SV models with Gamma-type log-volatility. In
particular, we show that for any finite number of given coefficients of
tail dependence $\eta_h \in[1/2,1]$ for the pairs $(X_0,X_h), 1 \leq
h \leq m$, one can find an SV model of the new type with exactly these
characteristics. This result underpins the high flexibility of our
approach to
modeling the extremal dependence of consecutive returns.
Most proofs are postponed to Section~\ref{proofsection}.

Throughout the paper, we use the following notation. We write $\mathcal
{B}(M)$ for the Borel $\sigma$-algebra on the metric space $M$. Weak
convergence is denoted by $\stackrel{w}{\rightarrow}$. The expression
$\delta_x$ stands for the Dirac measure at $x$. Bold expressions like
$\mathbf{0}, \mathbf{1}$, etc. stand for a vector of suitable
dimension consisting of $0$'s, $1$'s, and so on. This convention
extends to sets, for example, for a vector $\mathbf{x}=(x_1, \ldots, x_d)$
we have $(\mathbf{x}, \bolds{\infty})=\bigtimes_{i=1}^d
(x_i,\infty)$. For a vector $\mathbf{x}$ or a matrix $\mathbf{A}$ we
denote its transposed with $\mathbf{x}'$ and $\mathbf{A}'$,
respectively. We denote the positive and negative part of $x \in
\mathbb{R}$ by $x^+:=\max\{x,0\}$ and $x^-:=-\min\{x,0\}$,
respectively. The complement of a set $A$ is denoted by~$A^c$. The
empty product, $\prod_{i \in\varnothing} X_i$, is by convention equal
to 1. Finally, $f(x)\sim g(x)$ means that $\lim_{x \to\infty}f(x)/g(x)=1$.

\section{A Breiman-type result for regular variation on \texorpdfstring{$(0,\infty)^d$}{(0,infty)d}}
\label{sectBreiman}

As explained in the \hyperref[sec1]{Introduction}, we will analyze the extremal
dependence of the volatilities $(\sigma_{t_1},\ldots,\sigma_{t_d})$
and of the returns $(X_{t_1}, \ldots, X_{t_d})$ of multiplicative time
series as in (\ref{eqmultimodel}) using the notion of regular
variation on the cones $[0,\infty)^d\setminus\{\mathbf{0}\}$ and $\E
^d=(0,\infty)^d$. Although we are mainly interested in the case $d=2$,
for the time being we allow for an arbitrary $d\in\N$.

Since in SV models we have
%
\begin{equation}
\label{eqretrep} \pmatrix{ X_{t_1}
\cr
\vdots
\cr
X_{t_d}}=
\pmatrix{ \sigma_{t_1} & \ldots& 0
\cr
\vdots& \ddots& \vdots
\cr
0 &
\cdots& \sigma_{t_d}} \pmatrix{ \varepsilon_{t_1}
\cr
\vdots
\cr
\varepsilon_{t_d}}= \pmatrix{ \varepsilon_{t_1} & \ldots& 0
\cr
\vdots& \ddots& \vdots
\cr
0 & \cdots& \varepsilon_{t_d}} \pmatrix{
\sigma_{t_1}
\cr
\vdots
\cr
\sigma_{t_d}},
\end{equation}
multivariate Breiman-type results are very useful to establish regular
variation of $(X_{t_1}, \ldots, X_{t_d})$. For regular
variation on $[0,\infty)^d\setminus\{\mathbf{0}\}$, such a result
has been stated in Basrak, Davis and Mikosch \cite{BaDaMi02}, Proposition A.1. If $\mathbf{X}$ is regularly varying on $[0,\infty
)^d\setminus\{\mathbf{0}\}$ with index $-\tilde{\alpha}<0$ and
$\mathbf{A}$ is a random matrix independent of $\mathbf{X}$ such that
$E(\llVert \mathbf{A}\rrVert _{\mbox{\scriptsize{op}}}^{\tilde{\alpha
}+\eps
})<\infty$ for some $\eps>0$, then $\mathbf{AX}$ is regularly
varying on $[0,\infty)^d\setminus\{\mathbf{0}\}$ as well with the
same index of regular variation. Here, $\llVert \cdot\rrVert _{\mbox
{\scriptsize
{op}}}$ denotes the operator norm for matrices, which is defined by
%
\begin{equation}
\label{Eqopnorm} \llVert\mathbf{A}\rrVert_{\mbox{\scriptsize
{op}}}=\sup
_{\mathbf{x} \in\mathbb{R}^d\dvtx  \llVert \mathbf{x}\rrVert =1}\llVert
\mathbf{Ax}\rrVert=\sup_{\mathbf{x} \in\mathbb{R}^d\dvtx  \llVert \mathbf
{x}\rrVert
=1}d
\bigl(\mathbf{Ax}, \{\mathbf{0}\}\bigr),
\end{equation}
where $d(\mathbf{x},B):=\inf_{\mathbf{y} \in B}\llVert \mathbf
{x}-\mathbf
{y}\rrVert $ for $\mathbf{x} \in\mathbb{R}^d, B \subset\mathbb{R}^d$
denotes the usual distance function induced by the Euclidean norm on
$\mathbb{R}^d$. More precisely, if $\mathbf{Z}=\mathbf{X}$ satisfies
(\ref{eqmultiregvar}), then
\[
\frac{P(\mathbf{AX}\in xB)}{P(\llVert \mathbf{X}\rrVert >x)} \to
E\bigl(\tilde\nu\bigl(\mathbf{A}^{-1} (B)\bigr)
\bigr),\qquad x \to\infty, %
\]
for all $B\in\mathcal{B}([0,\infty)^d)$ bounded away from $\mathbf
{0}$ with $E(\tilde\nu(\mathbf{A}^{-1}(\partial B)))=0$.
Therefore, if the vector $(\eps_{t_1}, \ldots, \eps_{t_d})$ of
i.i.d. innovations is regularly varying on $[0,\infty)^d\setminus\{
\mathbf{0}\}$ with index $-\tilde{\alpha}$ and $E(\sigma_0^{\tilde
{\alpha}+\eps})<\infty$, then $(X_{t_1}, \ldots, X_{t_d})$ is
regularly varying on $[0,\infty)^d\setminus\{\mathbf{0}\}$ as well.
Likewise, one may draw this conclusion if the vector $(\sigma
_{t_1},\ldots,\sigma_{t_d})$ is regularly varying on $[0,\infty
)^d\setminus\{\mathbf{0}\}$ with index $-\tilde{\alpha}$ and
$E(\llvert \eps_0\rrvert ^{\tilde\alpha+\eps})<\infty$ for some $\eps>0$.

In the following, we will derive an analogous result for regular
variation on the smaller cone~$\mathbb{E}^d$. For the case of an SV
model with heavy-tailed i.i.d. innovations which are regularly varying
with index $-\alpha$ and light-tailed volatilities which satisfy
$E(\sigma_0^{2d\alpha+\varepsilon})<\infty$ for some $\varepsilon>0$,
Kulik and Soulier \cite{KuSo13a} have shown that the vector
$(X_{t_1}, \ldots, X_{t_d})$, $t_1<\cdots<t_d$, shows regular
variation on $\mathbb{E}^d$ with limiting measure
%
\begin{equation}
\label{EqSVstandard} \nu^X\Biggl(\bigtimesd_{i=1}^d(s_i,
\infty)\Biggr)=\frac{E (\nu^\varepsilon((\bolds{\Delta
}(\sigma_{t_1}, \ldots, \sigma_{t_d}))^{-1}(\bigtimes_{i=1}^d(s_i,
\infty)) ) )}{E (\nu^\varepsilon((\bolds
{\Delta}(\sigma_{t_1}, \ldots, \sigma_{t_d}))^{-1}(\bigtimes
_{i=1}^d(1, \infty)) ) )}=\prod_{i=1}^d
s_i^{-\alpha},
\end{equation}
for all $s_i>0, 1 \leq i \leq d$, where $\nu^\varepsilon$ with $\nu
^\varepsilon(\bigtimes_{i=1}^d(s_i, \infty))=\prod_{i=1}^d s_i^{-\alpha
}$ denotes the limiting measure of $(\varepsilon_{t_1}, \ldots, \varepsilon
_{t_d})$ on $\mathbb{E}^d$ and $\bolds{\Delta}(a_1, \ldots,
a_d)$ stands for a diagonal matrix with diagonal elements $a_1, \ldots, a_d$.

We want to derive a general Breiman-type result for regular variation
on the cone $\E^d$ for random products $\mathbf{A}\mathbf{X}$.
Because convergence (\ref{vagueconvergence}) describes the behavior of
random vectors exclusively on $\E^d$, we have to ensure that the
random pre-image $\mathbf{A}^{-1} (B)$ of a set $B\subset\E^d$ is
again a subset of $\E^d$. Moreover, as the limit measure $\nu$ in
(\ref{vagueconvergence}) may be infinite in a neighborhood of the
boundary $\OO^d$ of $\E^d$,
for our Breiman-type result we must control the distance between
$\mathbf{A}^{-1} (B)$ and $\OO^d$.

To this end, set $\mathbb{F}^d:=\R^d \setminus\mathbb{E}^d$ and
introduce the notation
\[
\tau(\mathbf{x}):=d\bigl(\mathbf{x}, \mathbb{F}^d\bigr)=\min\bigl(
\mathbf{x}^+\bigr)\qquad\mbox{for } \mathbf{x} \in\mathbb{R}^d,
\]
where $(x_1, \dots, x_d)^+=(x_1^+, \dots, x_d^+)$. Furthermore, let
\[
\mathcal{S}^d:= \bigl\{\mathbf{x} \in\mathbb{E}^d \mid
\tau(\mathbf{x})=1\bigr\}. %
\]
Observe that in the definition of regular variation on $\E^d$ the set
$\OO^d$ takes over the role which is played by the origin $\mathbf
{0}\in\mathbb{R}^d$ in the definition of regular variation on
$[0,\infty)^d \setminus\{\mathbf{0}\}$. Since $\tau(\mathbf
{x})=d(\mathbf{x},\OO^d)$
for all $\mathbf{x}\in\E^d$, one may consider $\mathcal{S}^d$ an
analog to the unit sphere $\{\mathbf{x} \in\R^d \mid\llVert \mathbf
{x}\rrVert
=d(\mathbf{x}, \mathbf{0})=1\}$ in the present setting. Cf. Das, Mitra
and Resnick \cite{DaMiRe11} for the use of $\mathcal
{S}^d$ in the context of regular variation on cones.\vadjust{\goodbreak}

Extend the definition of $\tau$ to matrices in analogy to (\ref
{Eqopnorm}) by setting
%
\begin{equation}
\label{Eqtaudef} \tau(\mathbf{A}):=\tau_{\mbox{\scriptsize{op}}}(\mathbf
{A}):=\sup
_{\mathbf{x} \in
\mathcal{S}^d} \tau(\mathbf{Ax})\qquad\mbox{for } \mathbf{A} \in
\mathbb{R}^{d \times d},
\end{equation}
with $\tau_{\mbox{\scriptsize{op}}}(\mathbf{A}) \in[0, \infty]$.
For brevity, we will write $\tau(\mathbf{A})$ instead of $\tau
_{\mbox{\scriptsize{op}}}(\mathbf{A})$ whenever it is clear that
$\mathbf{A}$ denotes a matrix. It will turn out that one may prove a
Breiman-type result for the regular variation on $\E^d$ if one
replaces the moment condition $E(\llVert \mathbf{A}\rrVert _{\mbox
{\scriptsize
{op}}}^{\tilde{\alpha}+\eps})<\infty$ used in Basrak, Davis and Mikosch
\cite{BaDaMi02} with the corresponding condition on
$\tau(\mathbf{A})$.
The next lemma provides some help for the interpretation of $\tau$.

%
\begin{lemma} \label{lemmatauprop}
For a matrix $\mathbf{A}\in\R^{d\times d}$ the following three
properties are equivalent:
\begin{longlist}[(iii)]
\item[(i)] $\mathbf{A}^{-1}(\E^d)\subset\E^d$;

\item[(ii)]$0<\tau(\mathbf{A})<\infty$;

\item[(iii)] $\mathbf{A}$ is invertible and all entries of $\mathbf{A}^{-1}$
are non-negative.
\end{longlist}
Furthermore, if $\mathbf{A}$ possesses any of these properties, then
$\mathbf{A}^{-1}([0,\infty)^d)\subset[0,\infty)^d$.
\end{lemma}

For a $d$-dimensional diagonal matrix $\bolds{\Delta}(\delta_1,
\ldots, \delta_d)$ with positive diagonal elements the value of $\tau$ may be directly derived from
(\ref{Eqtaudef}) as
\[
\tau(\bolds{\Delta})=\max\{\delta_1, \dots,
\delta_d\}.
\]
For a general random matrix $\mathbf{A,}$ the concrete value of $\tau
(\mathbf{A}) \in(0, \infty)$ can be derived from the inverse matrix
$\mathbf{A}^{-1}$.

%
\begin{lemma}\label{twolemmas}
Let $\mathbf{A} \in\mathbb{R}^{d \times d}$ satisfy $0<\tau(\mathbf
{A})<\infty$. Then
%
\begin{equation}
\label{Eqtauinverse} \tau(\mathbf{A})= \frac
{1}{\min(\mathbf{A}^{-1}\mathbf{1})}.
\end{equation}
\end{lemma}

We are now ready to state our Breiman-type result for vectors that show
regular variation on the cone $\E^d$.

\begin{theorem}\label{hiddenBreiman}
Let $\mathbf{Z} \in\R^d$ be regularly varying on the cone $\E^d$
with index $-\alpha$ for some $\alpha>0$ and limit measure $\nu$.
Moreover, let $\mathbf{A} \in\mathbb{R}^{d \times d}$ be a random
matrix independent of $\mathbf{Z}$ which satisfies $\tau(\mathbf
{A})>0$ almost surely and\vspace*{-2pt}
%
\begin{equation}
\label{finitemoment} E\bigl(\tau(\mathbf{A})^{\alpha
+\delta}\bigr)<\infty
\end{equation}
for some $\delta>0$. Then
%
\begin{equation}
\label{theorem} \lim_{x \to\infty}\frac{P(\mathbf
{AZ} \in xB)}{P(\min(\mathbf{Z})>x)}=E\bigl(\nu\bigl(
\mathbf{A}^{-1}B\bigr)\bigr)<\infty
\end{equation}
for all Borel sets $B \subset\mathbb{E}^d$ bounded away from $\OO^d$
with $E(\nu(\partial(\mathbf{A}^{-1}B)))=0$. In particular, $\mathbf
{AZ}$ is regularly varying on $\E^d$ with index $-\alpha$ and
limiting measure
\[
B \mapsto\frac{E(\nu(\mathbf{A}^{-1}B))}{E(\nu(\mathbf
{A}^{-1}((1,\infty)^d)))}. %
\]
\end{theorem}


The above theorem is useful to derive the extremal behavior of
stochastic volatility models.
%
\begin{example} \label{cormultSV}
Let $(X_{1,t},X_{2,t})_{t \in\mathbb{Z}}$ be a two-dimensional
heavy-tailed stochastic volatility time series with negative dependence
which satisfies
\begin{eqnarray*}
(X_{1,t},X_{2,t})'&=&\mathbf{L}_t
\mathbf{G}_t (\varepsilon_{1,t}, \varepsilon_{2,t})',\qquad t \in\mathbb{Z},
\\
\mathbf{L}_t &=& \pmatrix{ 1 & 0
\cr
-\exp(q_t/2) & 1},\qquad
\mathbf{G}_t= \pmatrix{ \exp(h_{1,t}/2) & 0
\cr
0 &
\exp(h_{2,t}/2)},\qquad t \in\mathbb{Z},
\\
q_{t+1}&=&\alpha+ \beta q_t + u_t,\qquad t \in
\mathbb{Z} %
\\
h_{i,t+1}&=&\mu_i + \phi_i h_{i,t} +
\eta_{i,t},\qquad i=1, 2, t \in\mathbb{Z}, %
\end{eqnarray*}
with $\beta, \phi_i \in(-1,1)$ and i.i.d. $\varepsilon_{i,t}$ which
have a Student-$t$-distribution with $\alpha$ degrees of freedom,
$\eta_{i,t}\sim\mathcal{N}(0,\sigma_{\eta,i}^2), u_t\sim\mathcal
{N}(0,\sigma_u^2)$, all mutually\vspace*{2pt} independent. Cf. Tsay
\cite{Ts02}, Section~9.6, and Asai, McAleer and Yu \cite
{AsMcYu06}, Section~2.5.2, for similar models. Define independent
random variables
\[
H_i \sim\mathcal{N}\bigl(\mu_i, \sigma_{\eta,i}^2/
\bigl(1-\phi_i^2\bigr)\bigr),\qquad i=1,2,\qquad Q \sim\mathcal{N}
\bigl(\alpha, \sigma_{u}^2/\bigl(1-\beta^2
\bigr)\bigr).
\]
The stationary distribution of $\mathbf{L}_t \mathbf{G}_t$ is equal
to the distribution of $\mathbf{A}$ with
\begin{eqnarray*}
\mathbf{A} &=&\pmatrix{ \exp(H_1/2) & 0
\vspace*{3pt}\cr
-\exp
\bigl((Q+H_1)/2\bigr) & \exp(H_2/2)},
\\
\mathbf{A}^{-1} &=& \pmatrix{ \exp(-H_1/2) & 0
\vspace*{3pt}\cr
\exp
\bigl((Q-H_2)/2\bigr) & \exp(-H_2/2)}. %
\end{eqnarray*}
Since all entries of $\mathbf{A}^{-1}$ are a.s. non-negative, we may
apply Theorem~\ref{hiddenBreiman} to see that the stationary
distribution of $(X_{1,t},X_{2,t})$ is regularly varying on $\mathbb
{E}^2$ with index $-2\alpha$ and limit measure
\begin{eqnarray*}
\nu_2^{X} \Biggl(\bigtimesd_{i=1}^2(s_i,
\infty) \Biggr)&=& \frac{E ( (\exp(-H_1/2)s_1 )^{-\alpha} (\exp
((Q-H_2)/2)s_1
+\exp(-H_2/2)s_2 )^{-\alpha} )}{E ( (\exp
(-H_1/2) )^{-\alpha}
(\exp((Q-H_2)/2)+\exp(-H_2/2) )^{-\alpha} )}
\\
&=& s_1^{-\alpha}\frac{E ( (\exp((Q-H_2)/2)s_1
+\exp(-H_2/2)s_2 )^{-\alpha} )}{E ( (\exp((Q-H_2)/2)
+\exp(-H_2/2) )^{-\alpha} )}
\end{eqnarray*}
for $s_1,s_2>0$.
\end{example}
In the following, we are mainly interested in the behavior of lagged
observations of univariate SV models. That all SV models with
heavy-tailed i.i.d. innovations and light-tailed volatilities have the
same limit measure (\ref{EqSVstandard}) has been shown in Kulik and
Soulier \cite{KuSo13a} and also follows from Theorem~\ref
{hiddenBreiman} (even under the weaker assumption that $E(\sigma
_0^{d\alpha+\varepsilon})<\infty$ instead of $E(\sigma_0^{2d\alpha
+\varepsilon})<\infty$\vadjust{\goodbreak} for some $\varepsilon>0$). The following lemma
shows that a much richer second-order structure evolves from SV models
with heavy-tailed volatilities and light-tailed innovations.

%
\begin{corollary} \label{corSVcor}
Let $X_t=\sigma_t\eps_t$, $t\in\Z$, where $(\sigma_t)_{t\in\Z}$
and $(\eps_t)_{t\in\Z}$ are independent stationary time series with
$\eps_t$, $t\in\Z$, i.i.d. and $\sigma_t>0$. Assume furthermore
that $(\sigma_{t_1},\ldots, \sigma_{t_d})$ is regularly varying on
$\E^d$ with index $-\alpha_d<0$ and limit measure $\nu_d^\sigma$.
Let, in addition, $0<E((\varepsilon_0^+)^{\alpha_d+\delta})<\infty$
for some $\delta>0$. Then $(X_{t_1},\ldots,X_{t_d})$ is regularly
varying on $\E^d$ as well, with the same index $-\alpha_d$ and limit
measure $\nu_d^{X}$ defined by
%
\begin{equation}
\label{hiddenloggamma} \nu_d^{X} \Biggl(\bigtimesd_{i=1}^d(s_i,
\infty) \Biggr) = \frac{E
(\nu_d^\sigma(\bigtimes_{i=1}^d (s_i (\varepsilon
_{t_i}^+ )^{-1}, \infty) ) )}{
E (\nu_d^\sigma(\bigtimes_{i=1}^d ( (\varepsilon
_{t_i}^+ )^{-1}, \infty) ) )}
\end{equation}
for all $s_1,\ldots,s_d>0$.
\end{corollary}

\begin{pf}
Let $\eps_{t_i}^*$, $1\le i\le d$, be i.i.d. random variables
independent of $(\sigma_{t_1},\ldots, \sigma_{t_d})$ with
distribution $P^{\eps_0\mid\eps_0>0}$ and define $X_{t_i}^\ast:=\varepsilon
_{t_i}^\ast\sigma_{t_i}, i=1, \ldots, d$. Apply Theorem
\ref{hiddenBreiman} to the second equation in~(\ref{eqretrep}) to
conclude the regular variation of $(X_{t_1}^*,\ldots,X_{t_d}^*)$ on
$\E^d$ with limit measure $\nu_d^{X}$, which is equivalent to the assertion.
\end{pf}
An analogous result holds true if $X_t$ is replaced by $\llvert
X_t\rrvert $.

In view of Corollary~\ref{corSVcor}, we are able to construct SV
models with a flexible second-order extremal behavior of the lagged
log returns as soon as we find a way to model the volatilities
accordingly. A suitable model will be introduced in the next section.

\section{SV models with Gamma-type log-volatilities}\label{sectSVmodels}

In order to allow for a richer second-order extremal dependence
structure of SV models we modify the assumption of a normal
distribution of the log-volatilities in a way which allows for heavy
tails of the volatility process. In our construction, we rely on
results from Rootz{\'e}n \cite{Ro86} which guarantee the existence
of stationary time series that meet our assumptions.

%
\begin{definition} \label{defSVnew}
Let
%
\begin{equation}
\label{Xdef} X_t=\sigma_t \varepsilon_t,\qquad t
\in\mathbb{Z},
\end{equation}
with $\varepsilon_t, t \in\mathbb{Z}$, i.i.d. such that $P(\varepsilon
_0>0)>0$ and $E(\llvert \varepsilon_0\rrvert ^{1+\delta})<\infty$ holds
for some
$\delta>0$. Furthermore, let
%
\begin{equation}
\label{sigmadef} \log\sigma_t=\sum_{i=0}^\infty
\alpha_i \eta_{t-i},\qquad t \in\mathbb{Z},
\end{equation}
with:
\begin{longlist}[(a)]
\item[(a)] coefficients $\alpha_i \in[0,1], i \in\mathbb{N}_0$,
such that
$\max_{i \in\mathbb{N}_0}\alpha_i=1$ and $\alpha_i = \mathrm{O}
(i^{-\theta})$ as $i \to\infty$ for some $\theta>1$,
\item[(b)] i.i.d. innovations $\eta_t, t \in\mathbb{Z}$, which are
independent of $(\varepsilon_t)_{t \in\mathbb{Z}}$ with $E(\eta
_0^2)<\infty$ and
%
\begin{equation}
\label{etaasymptotic} P(\eta_0>z) \sim K z^\beta{\rm
{e}}^{-z},\qquad z \to\infty,
\end{equation}
for a real constant $\beta\neq-1$ and a positive constant $K$.
\end{longlist}
We call $(X_t, \sigma_t)_{t \in\mathbb{Z}}$ a \emph{stochastic
volatility} (\emph{SV}) \emph{model with Gamma-type log-volatility}.
\end{definition}

In accordance with common SV terminology, we will call the process
$(\sigma_t)_{t \in\mathbb{Z}}$ the \textit{volatility process},
$(\log\sigma_t)_{t \in\mathbb{Z}}$ the \textit{log-volatility
process} and $(\varepsilon_t)_{t \in\mathbb{Z}}$ the \textit
{innovation process}. The following theorem is based on Rootz{\'e}n
\cite{Ro86} and guarantees that a heavy-tailed stationary
solution to our definition exists.

%
\begin{theorem} \label{thstatsol}
There exists a stationary solution $(X_t, \sigma_t)_{t \in\mathbb
{Z}}$ to (\ref{Xdef}) and (\ref{sigmadef}) as in Definition~\ref
{defSVnew} and the marginal distributions of $\llvert X_0\rrvert $
and $\sigma_0$
are regularly varying with index $-1$. Furthermore, the distribution of
$X_0$ is tail balanced with
%
\begin{equation}
\label{tail-balanced}\lim_{x \to\infty}\frac
{P(X_0>x)}{P(\llvert X_0\rrvert >x)}=\frac{E(\eps_0^+)}{E(\llvert
\eps_0\rrvert )},
\qquad\lim_{x \to\infty}\frac{P(X_0<-x)}{P(\llvert X_0\rrvert
>x)}=\frac{E(\eps
_0^-)}{E(\llvert \eps_0\rrvert )}
\end{equation}
and
%
\begin{equation}
\label{eqXtails} \lim_{x \to\infty}\frac{P(\llvert X_0\rrvert
>x)}{P(\sigma_0>x)}=E\bigl(\llvert
\eps_0\rrvert\bigr), \qquad\lim_{x \to\infty}
\frac{P(X_0>x)}{P(\sigma_0>x)}=E\bigl(\eps_0^+\bigr).
\end{equation}
For normalizing constants
$a_n:=\hat{K} n (\log n)^{\hat{\beta}}, n \in\mathbb{N}$ [see
(\ref{betahat}) and (\ref{Khat}) for the definition of $\hat{\beta
}$ and $\hat{K}$] and $z>0$ we have
%
\begin{eqnarray}
\label{independentmaxima}  P(\sigma_0\leq a_n z)^n
&\to&\exp(-1/z), \qquad P\bigl(\llvert X_0\rrvert\leq a_n
z\bigr)^n \to\exp\bigl(-E\bigl(\llvert\varepsilon_0
\rrvert\bigr)/z\bigr),
\nonumber\\[-8pt]\\[-8pt]\nonumber
P(X_0\leq a_n z)^n &\to&\exp
\bigl(-E\bigl(\varepsilon_0^+\bigr)/z\bigr),\qquad n \to\infty.
\end{eqnarray}
\end{theorem}
%

\begin{remark}\label{remstandardization}
(i)~The assumption $\max_{i\in\N_0}\alpha_i=1$ ensures that the
index of regular variation equals $-1$. However,
our model can easily be extended to an arbitrary (negative) index of
regular variation. To this end, replace (\ref{sigmadef}) with
\[
\log\sigma_t=c \sum_{i=0}^\infty
\alpha_i \eta_{t-i}, \qquad t\in\Z, %
\]
for some $c>0$. Together with the above assumptions this will lead to a
solution of (\ref{sigmadef}) which is regularly varying with index
$-1/c$. If we assume that $E(\llvert \varepsilon_0\rrvert ^{1/c+\delta
})<\infty$ for
some $\delta>0$, then again by Breiman's lemma this implies a
stationary solution to (\ref{Xdef}) which is regularly varying with
the same index $-1/c$. For the sake of notational simplicity, we will
stick to the original definition in the following analysis.\vadjust{\goodbreak}

(ii)~Stochastic volatility models often have an additional parameter
specifying the mean of the log-volatilities (cf., e.g., Taylor \cite
{Ta86}). In this case, equation (\ref{sigmadef})
would read as $\log\sigma_t=\mu+\sum_{i=0}^\infty\alpha_i \eta
_{t-i}$ for some $\mu\in\mathbb{R}$. Such an assumption is usually
combined with the standardization of some moment of $\varepsilon_0$, for
example, by setting $\operatorname{Var}(\varepsilon_0)=1$. Otherwise, setting
$\hat{\varepsilon}_i:={\rm{e}}^\mu\varepsilon_i, i \in\mathbb{Z}$, has
the same effect as adding $\mu$ in the definition of $\log\sigma_t$.
Since we make no assumptions about the particular form of (existing)
moments of $\varepsilon_0$, we set $\mu=0$ without loss of generality.

(iii)~The case where $\beta=-1$ in (\ref{etaasymptotic}) is a
boundary case which is not treated in detail in Rootz{\'e}n~\cite{Ro86}, therefore, it is excluded in Definition~\ref{defSVnew}.
\end{remark}

%
\begin{example}\label{exAR1}
An interesting special case is given by $\alpha_i:=\alpha^i, i \in
\mathbb{N}_0$, for some $\alpha\in(0,1)$. This case corresponds to
an AR(1) model for the log-volatilities, that is,
%
\begin{equation}
\label{AR1} \log\sigma_t =\alpha\log\sigma_{t-1}+
\eta_t, \qquad t \in\mathbb{Z}.
\end{equation}
A similar model, with a modified assumption about the distribution of
$\eta_t, t \in\mathbb{Z}$ (namely, that the distribution of $\exp
(\eta_t)$ is regularly varying and that a stationary solution to (\ref
{AR1}) exists) has been analyzed with respect to its first-order
extremal behavior in Mikosch and Rezapour \cite{MiRe12}.

Moreover, the conditional extreme value behavior of consecutive
observations given that $X_0$ is large has been analyzed by Kulik and
Soulier \cite{KuSo13} in the case that the innovations $\eta_t$
of the log-volatility series are double exponentially distributed. See
Section~\ref{sectasymptoticsSV} for a more detailed comparison with
their results.
\end{example}

Next, we are interested in the extremal behavior of the \textit
{processes} $(X_t)_{t \in\mathbb{Z}}$ and $(\sigma_t)_{t \in\mathbb
{Z}}$, particularly in their extremal dependence structure. Some
information on their first-order extremal dependence behavior may
readily be derived from the point process results in Rootz{\'e}n
\cite{Ro86} for the process of the log-volatilities. In the
following, let $M_p(\mathbb{E})$ denote the set of Radon point
measures on a topological space $\mathbb{E}$. For an introduction to
point processes in the context of extreme values, see Resnick \cite
{Re07}, Chapters~5 and 7.

%
\begin{theorem}\label{pointprocesses}
Let $(X_t, \sigma_t)_{t \in\mathbb{Z}}$ be an SV model with
Gamma-type log-volatility. In the case that $\beta<-1$ assume
additionally that $k:=\llvert \{n \in\mathbb{N}_0\dvtx  \alpha_n=1\}
\rrvert =1$. With
$a_n, n \in\mathbb{N}$, as in Theorem~\ref{thstatsol}, let
$N_n^{\sigma}, n \in\mathbb{N}$, and $N_n^{X}, n \in\mathbb{N}$,
denote the point processes defined by
\[
N_n^{\sigma}(\cdot):=\sum_{i=1}^n
\delta_{(i/n, \sigma
_i/a_n)}(\cdot),\qquad N_n^X(\cdot):=\sum
_{i=1}^n \delta_{(i/n,
X_i/a_n)}(\cdot).
\]
\begin{longlist}[(iii)]
\item[(i)] Then, as $n \to\infty$, $N_n^\sigma\stackrel{w}{\rightarrow}
N^\sigma$ in $M_p([0, 1] \times(0, \infty])$, where $N^\sigma$ is a
Poisson process with intensity measure $\mathrm{d}t \times z^{-2}\,\mathrm{d}z$.

\item[(ii)] Let $(t_{(i)},z_{(i)})_{i \in\mathbb{N}}$ denote the points of
the Poisson process $N^\sigma$ and $(\varepsilon_{(i)})_{i \in\mathbb
{N}}$ an i.i.d. sequence with $P^{\varepsilon_{(1)}}=P^{\varepsilon_0}$,
independent of $(t_{(i)},z_{(i)})_{i \in\mathbb{N}}$.\vadjust{\goodbreak}

Then $N_n^X (\cdot\cap D) \stackrel{w}{\rightarrow} N^X(\cdot\cap
D)$ in $M_p(D)$, as $n \to\infty$, where $D:=[0, 1]
\times[-\infty, \infty]\setminus\{0\}$ and $N^X$ is a point process\vspace*{1.5pt}
consisting of the points $(t_{(i)},z_{(i)}\varepsilon_{(i)})_{i \in
\mathbb{N}}$. The restriction of $N^X$ to $D$ is a Poisson point
process with intensity measure $\mathrm{d}t \times z^{-2} [\mathds{1}_{\{
z<0\}}E(\varepsilon_0^-)+\mathds{1}_{\{z>0\}}E(\varepsilon_0^+) ]\,\mathrm{d}z$.
\end{longlist}
\end{theorem}

From part (i) of Theorem~\ref{pointprocesses}, we may conclude that
\[
P\bigl(\max\{\sigma_1, \dots, \sigma_n\}\leq
a_n z\bigr)=P \bigl(N_n^\sigma\bigl([0,1]
\times(z,\infty)\bigr)=0 \bigr)\to\exp(-1/z)
\]
for all $z>0$. A comparison with (\ref{independentmaxima}) thus yields
that the extremal index of the stationary sequence $(\sigma_t)_{t \in
\mathbb{Z}}$ exists and equals 1. The same holds true for the
processes $(\llvert X_t\rrvert )_{t \in\mathbb{Z}}$ and $(X_t)_{t \in
\mathbb{Z}}$
by part (ii) of the theorem. Hence, the extremes in these processes do
not cluster asymptotically, in the sense that the projections of the
limiting point processes obtained in
Theorem~\ref{pointprocesses} on the time coordinate are simple (i.e.,
they do not have multiple points), cf. Leadbetter \cite
{Le83}. Moreover, the form of the limiting point process implies
asymptotic independence, that is, %
\[
P(\sigma_h>x \mid\sigma_0>x)\to0\quad\mbox{and}\quad
P(X_h>x \mid X_0>x)\to0
\]
as $x \to\infty$ for all $h \in\mathbb{Z} \setminus\{0\}$.

Hence, in this respect, the SV model with Gamma-type log-volatility
shows the same first-order extremal dependence behavior as classical SV
time series.
In Section~\ref{sectasymptoticsSV}, though, we will see that the second-order extremal
behavior of these classes of processes is quite different. In what
follows, we focus on the asymptotics for the probabilities $P(\sigma
_0>s_0x, \sigma_h>s_hx)$ and $P(X_0>s_0x, X_h>s_hx)$ as $x \to\infty
$ for different values of $h \in\mathbb{N}$ and $s_0,s_h>0$. For
simplicity, we will restrict ourselves to the analysis of the upper
tails of the process $(X_t)_{t \in\mathbb{Z}}$. However, the
necessary changes to analyze both upper and lower tails become obvious
by writing $P(X_0>s_0x, X_h<-s_hx)=P(\sigma_0\varepsilon_0^+>s_0x,
\sigma_h\varepsilon_h^->s_hx)$ for $s_0,s_h,x>0$.

\section{Joint extremal behavior of power products}\label{sectpowerproducts}

In this section, we analyze the joint extremal behavior of products of
the form $Y_i=\prod_{j=1}^\infty X_j^{\alpha_{ij}}$, $1 \leq i \leq d$,
for i.i.d. non-negative random variables $X_j, j \in\mathbb{N}$,
which are regularly varying with index $-1$. The connection to SV
models with Gamma-type log-volatilities becomes clear by writing
\[
\sigma_t=\exp\Biggl(\sum_{i=0}^\infty
\alpha_i \eta_{t-i} \Biggr)=\prod
_{i=0}^\infty\bigl(\exp(\eta_{t-i})
\bigr)^{\alpha_i},\qquad t \in\mathbb{Z}.
\]
Most of this section deals with the case $d=2$ and we will write $\prod
_{i=1}^\infty X_i^{\alpha_i}$ and $\prod_{i=1}^\infty X_i^{\beta_i}$
for notational convenience. The general case will briefly be discussed
at the end of this section.
We will show that the joint tail behavior of these two products is
closely related to the following infinite-dimensional linear
optimization problem:
\[
\sum_{i=1}^\infty\kappa_i \to
\min!
\]
under the constraints
\[
\sum_{i=1}^\infty\alpha_i
\kappa_i\geq1,\qquad \sum_{i=1}^\infty
\beta_i\kappa_i\geq1,\qquad \kappa_i \geq0, \forall
i \in\mathbb{N}. %
\]

This relation can be explained by the following heuristic argument for $\alpha_i$, $\beta_i \geq 0$.
Suppose that $(\kappa_i)_{ i \in\mathbb{N}}$ is a sequence that
fulfills the constraints. Then the event $\{X_i> x^{\kappa_i}, i \in
\mathbb{N}\}$ implies both
\[
\prod_{i=1}^\infty X_i^{\alpha_i}>
\prod_{i=1}^\infty x^{\alpha
_i\kappa_i}\geq x
\]
and
\[
\prod_{i=1}^\infty X_i^{\beta_i}>
\prod_{i=1}^\infty x^{\beta
_i\kappa_i}\geq x
\]
for $x\geq1$. Now, $x\mapsto P(X_i>x^{\kappa_i}, i\in\N)$ is
regularly varying with index $-{\sum_{i=1}^\infty\kappa_i}$. Hence,
if $\sum_{i=1}^\infty\kappa_i$ is minimized, the above event is,
heuristically, the ``most likely'' combination of extremal events which
lead to $\{\prod_{i=1}^\infty X_i^{\alpha_i}>x,\prod_{i=1}^\infty
X_i^{\beta_i}>x\}$.

We will make frequent use of the so-called Potter bounds (Bingham {\it
et~al.} \cite{BiGoTe87}, Theorem~1.5.6) for functions
$f\dvtx [0,\infty)\to(0, \infty)$ which are regularly varying with index
$-\alpha$: for all $\varepsilon>0$ there exists a constant $M=M(\varepsilon
)$ such that if $\min\{x, sx\}>M$, then
%
\begin{equation}
\label{eqpotterbounds} (1-\varepsilon)f(x)s^{-\alpha\mp\varepsilon}\leq f(s
x)\leq(1+\varepsilon
)f(x)s^{-\alpha\pm\varepsilon},
\end{equation}
where, for abbreviation,
%
\begin{equation}
\label{eqpluminus} s^{\pm\varepsilon}:=\max\bigl\{ s^\varepsilon, s^{-\varepsilon}
\bigr\},\qquad s^{\mp\varepsilon}:=\min\bigl\{ s^\varepsilon, s^{-\varepsilon}\bigr\}.
\end{equation}

Before we deal with infinite products, we first analyze products of two
factors in the case that a unique solution to the above optimization
problem exists with $\kappa_1,\kappa_2>0$.

\begin{prop}\label{power-products}
Let $X_1, X_2 \geq0$ be two independent random variables which are
both regularly varying with index $-1$. For constants $\alpha_1,\alpha
_2,\beta_1,\beta_2\geq0$, we assume that the linear optimization problem
%
\begin{eqnarray}
\label{linearequations} &&\kappa_1+\kappa_2 \to\min!
\nonumber\\[-8pt]\\[-8pt]\nonumber
&&\quad \mbox{under the constraints } \kappa_1\alpha_1+
\kappa_2\alpha_2\geq1, \kappa_1\beta
_1+\kappa_2\beta_2\geq1,
\kappa_1, \kappa_2 \geq0
\end{eqnarray}
has a unique solution satisfying $\kappa_1,\kappa_2>0$.
Let
\[
C:=\frac{\llvert \alpha_1\beta_2-\alpha_2\beta_1\rrvert }{(\alpha
_1-\alpha
_2)(\beta_2-\beta_1)}.
\]
Then,
to each $\varepsilon>0$, there exists a constant $M=M(\varepsilon)>0$ such that
\begin{eqnarray*}
(1-\varepsilon)C z_1^{1\mp\varepsilon}z_2^{1\mp\varepsilon} &
\leq&\frac{P((z_1X_1)^{\alpha_1} (z_2X_2)^{\alpha
_2}>x,(z_1X_1)^{\beta_1}(z_2X_2)^{\beta_2}>x)}{P(X_1>x^{\kappa
_1})P(X_2>x^{\kappa_2})}
\\
& \leq& (1+\varepsilon)C z_1^{1\pm\varepsilon}z_2^{1\pm\varepsilon}
\end{eqnarray*}
for all $x,z_1,z_2>0$ satisfying $\min\{x^{\kappa
_1},z_1^{-1}x^{\kappa_1}$, $x^{\kappa_2}, z_2^{-1}x^{\kappa_2}\}>M$.
\end{prop}

The above result can be interpreted as a bivariate extension of the
Potter bounds for random products and is essential for the proof of the
following theorem.
\begin{theorem}\label{power-products-infinite}
Let $X_i$, $i\in\N$, be i.i.d. non-negative random variables which
are regularly varying with index $-1$. Assume that $\alpha_i,\beta
_i\ge0$, $i\in\N$, are such that neither all $\alpha_i$ nor all
$\beta_i$ are equal to~0 and $\sum_{i=1}^\infty\alpha_i<\infty,
\sum_{i=1}^\infty\beta_i<\infty$.
Then
\[
Y_0:=\lim_{n \to\infty} \prod
_{i=1}^n X_i^{\alpha_i},\qquad
Y_1:=\lim_{n \to\infty} \prod
_{i=1}^n X_i^{\beta_i}
\]
exist almost surely. We assume that $P(Y_i>0)>0, i=0,1$.

Then the optimization problem
%
\begin{eqnarray}
\label{LOb} &&\sum_{i=1}^\infty
\kappa_i \to\min!
\nonumber\\[-8pt]\\[-8pt]\nonumber
&&\quad\mbox{under the constraints } \sum_{i=1}^\infty
\alpha_i \kappa_i \geq1, \sum _{i=1}^\infty\beta_i \kappa_i
\geq1, \kappa_i\ge0, \forall i\in\N,
\end{eqnarray}
has a solution $(\kappa_i)_{i\in\N}$.
\begin{longlist}[(iii)]
\item[(i)] For all $\eps>0$,
%
\begin{eqnarray}
\label{Equpperobound} P(Y_0>x,Y_1>x) & = & \mathrm{o}
\bigl(x^{\eps-\sum
_{i=1}^\infty\kappa_i} \bigr),
\\
\label{Eqlowerobound} x^{-\eps-\sum_{i=1}^\infty\kappa_i} & = & \mathrm{o}\bigl
(P(Y_0>x,Y_1>x)
\bigr)
\end{eqnarray}
as $x\to\infty$.
\end{longlist}
If the solution to (\ref{LOb}) is unique, then at most two of the
$\kappa_i$ are strictly positive.
\begin{enumerate}
\item[(ii)] If the solution is unique with $\kappa_i,\kappa_j>0$ for
some $i \neq j$, then for all $s_0, s_1>0$
%
\begin{equation}
\label{mproductsnondegenerateb} \lim_{x \to\infty} \frac{P(Y_0>s_0x,
Y_1>s_1x)}{P(X_i>x^{\kappa_i})P(X_j>x^{\kappa_j})} = D
\bigl(s_0^{\beta_i-\beta_j}s_1^{\alpha_j-\alpha_i}
\bigr)^{1/(\alpha_i\beta_j-\alpha_j\beta_i)}
\end{equation}
with
\[
D:=\frac{\llvert \alpha_i\beta_j-\alpha_j\beta_i\rrvert }{(\alpha
_i-\alpha
_j)(\beta_j-\beta_i)}E \biggl( \prod_{m\in\N\setminus\{i,j\}}
X_m^{(\alpha_m(\beta_j-\beta_i)+\beta_m(\alpha_i-\alpha
_j))/(\alpha_i\beta_j-\alpha_j\beta_i)} \biggr)<\infty.\vadjust{\goodbreak}
\]
\item[(iii)] If the solution is unique with $\kappa_i>0$ for exactly
one $i \in\N$ and $\kappa_j=0$ else, then for all $s_0,s_1>0$
%
\begin{eqnarray}
\label{mproductsdegenerate}
\nonumber
&&\lim_{x \to\infty}
\frac{P(Y_0>s_0x, Y_1>s_1
x)}{P(X_i>x^{\kappa_i})}
\\
&&\quad  =  \cases{ \displaystyle E \biggl( \min\biggl( s_0^{-1}
\prod_{j\in\N\setminus\{
i\}} X_j^{\alpha_j},
s_1^{-1}\prod_{j\in\N\setminus\{i\}}
X_j^{\beta_j} \biggr) ^{1/\beta_i} \biggr), &\quad $
\alpha_i=\beta_i$,
\vspace*{3pt}\cr
\displaystyle E \biggl( \prod
_{j\in\N\setminus\{i\}} X_j^{\beta
_j/\beta_i}
\mathds{1}_{\{X_j^{\alpha_j}>0\}} \biggr) s_1^{-1/\beta
_i}, & \quad$ \alpha_i>\beta_i$,
\vspace*{3pt}\cr
\displaystyle E \biggl( \prod
_{j\in\N\setminus\{i\}} X_j^{\alpha
_j/\alpha_i}
\mathds{1}_{\{X_j^{\beta_j}>0\}} \biggr) s_0^{-1/\alpha
_i}, &\quad$
\alpha_i<\beta_i$,}
\\
&&\quad <\infty.\nonumber
\end{eqnarray}
\end{enumerate}
\end{theorem}

%
\begin{remark} \label{remconvconds}
(i)~Under our assumptions about the summability of the $\alpha_i$'s
and $\beta_i$'s, the infinite-dimensional linear program (\ref{LOb})
can always be boiled down to a finite-dimensional linear program, which
can easily be solved numerically. Indeed, an optimal solution $(\kappa
_i)_{i\in\N}$ must satisfy $\kappa_i=0$ if $\max(\alpha_i,\beta
_i)< 1/(2/\max_{j\in\N}\alpha_j+2/\max_{j\in\N} \beta_j)$; see
part~2 of the proof of Theorem~\ref{power-products-infinite} for details.

(ii) If $E(\log X_1)^-<\infty$, then $Y_0$ and $Y_1$ are almost
surely strictly positive, because then $E(\sum_{i=1}^k (\alpha
_i+\beta_i)\log X_i)$ converges in $\R$ as $k\to\infty$. In
particular, it suffices to assume that $P(X_1\le x)=\mathrm{o}(\llvert \log
x\rrvert ^{-(1+\eps)})$ as $x\downarrow0$ for some $\eps>0$.

(iii) The assumption that all $X_j$ are identically distributed can be
dropped if one of the following two conditions is fulfilled:
\begin{itemize}
\item there exists $n \in\mathbb{N}$ such that $\alpha_i=\beta_i=0$
for all $i \geq n$,
\item$\sup_{i\in\N} E(X_i^\eps)<\infty$ for some $\eps>0$
(which ensures that the infinite products converge and a moment bound
of similar type as (\ref{Eqhelpfulbound}) holds).
\end{itemize}

(iv) If the solution to the linear program is not unique, it is
sometimes possible that a slight redefinition of the factors $X_i$
leads to linear program with a unique solution. Think,\vspace*{1pt} for example, of
$ \alpha_i=\beta_i=1, i=1,2$, $\sup_{i \geq3}\{\alpha_i, \beta_i\}
<1$. If we define $\hat{X}_1=X_1X_2,\hat{X}_i=X_{i+1}, i \geq2$,
then $\hat X_1$ is regularly varying with index $-1$ by the corollary
to Theorem~3 in Embrechts and Goldie \cite{EmGo80}. The resulting
linear program has a unique solution $\hat{\kappa}_1=1, \hat{\kappa
}_i=0, i\geq2$, and $P(Y_0>s_0x, Y_1>s_1x) \sim E(\min\{s_0^{-1}\prod
_{i\geq3}X_i^{\alpha_i},s_1^{-1}\prod_{i \geq3}X_i^{\beta_i}\}
)P(X_1X_2>x)$. However,\vspace*{1pt} the probability on the right-hand side cannot
be easily expressed in terms of tail probabilities of $X_1$ and $X_2$.
See Denisov and Zwart \cite{DeZw07} and Embrechts and Goldie
\cite{EmGo80} for discussions of the distribution of the product
of two factors with the same index of regular variation.
\end{remark}

%
%
\begin{remark}\label{remHRV}
(i)~In the situation of Theorem~\ref{power-products-infinite}(ii)
and (iii) $\min\{Y_0,Y_1\}$ is regularly varying with index $-\sum
_{i=1}^\infty\kappa_i$.

(ii)~Under the assumptions of Theorem~\ref{power-products-infinite}(ii) the random vector $(Y_0,Y_1)$ is regularly varying on the cone $\E
^2=(0,\infty)^2$ with limiting measure $\nu$ given by
\[
\nu\bigl((s_0,\infty)\times(s_1,\infty) \bigr)=
\bigl(s_0^{\beta
_i-\beta_j} s_1^{\alpha_j-\alpha_i}
\bigr)^{1/(\alpha_i\beta
_j-\alpha_j\beta_i)}, s_0,\qquad s_1>0. %
\]

(iii)~Under the assumptions of Theorem~\ref{power-products-infinite}(iii) with $\alpha_i=\beta_i$ the random vector $(Y_0,Y_1)$ is
regularly varying on the cone $\E^2$ with limiting measure $\nu$
given by
\[
\nu\bigl((s_0,\infty)\times(s_1,\infty) \bigr)=
\frac{E ( \min
( s_0^{-1}\prod_{j \in\mathbb{N} \setminus\{i\}} X_j^{\alpha
_j}, s_1^{-1}\prod_{j\in\mathbb{N} \setminus\{i\}} X_j^{\beta
_j} ) ^{1/\beta_i} )}{E ( \min( \prod_{j\in
\mathbb{N} \setminus\{i\}} X_j^{\alpha_j}, \prod_{j\in\mathbb{N}
\setminus\{i\}} X_j^{\beta_j} ) ^{1/\beta_i} )}, %
\]
$s_0, s_1>0$.

(iv) If the assumptions of Theorem~\ref{power-products-infinite}(iii) are fulfilled with $\alpha_i\neq\beta_i$, then there exists a
measure $\nu$ on $\mathcal{B}((0,\infty]^2)$ such that convergence
(\ref{vagueconvergence}) holds for all $\nu$-continuous Borel sets
$B\subset(0,\infty]^2$ bounded away from $\mathcal{O}^2$, but $\nu$
is concentrated on $(\{\infty\}\times(0,\infty))\cup((0,\infty
)\times\{\infty\})$. Thus, the restriction of $\nu$ to $\E^2$
equals the zero measure and $(Y_0,Y_1)$ is not regularly varying on~$\E
^2$. Note that this case cannot occur if $Y_0$ and $Y_1$ have the same
distribution, as it will be the case in the applications considered in
the next section.
\end{remark}

So far, we have analyzed the joint extremal behavior of only two
products of powers of the random variables $X_j$. In the remainder of
this section, we briefly discuss the joint extreme value behavior of an
arbitrary number of such products. However, as a generalization of the
Potter-type result established in Proposition~\ref{power-products} to
arbitrary dimension is rather cumbersome, we focus on a special case in
which one can use the one-dimensional Potter bounds instead.

In what follows, we consider products
\[
Y_i:= \prod_{j=1}^n
X_j^{\alpha_{ij}}, \qquad1\le i\le d, %
\]
where we assume w.l.o.g. that $d\le n$. (By the arguments given at the
beginning of the proof of Theorem~\ref{power-products-infinite}, one
can easily deal with infinite products, too.) For the main result in
this case, we do not assume any longer that all exponents $\alpha
_{ij}$ are non-negative, but that the sub-matrix consisting of the most
relevant exponents has a certain structure.

%
\begin{theorem}\label{power-products-high-dimension}
Assume that $X_j$, $1\le j\le n$, are independent, non-negative random
variables that are regularly varying with index $-1$ and bounded away
from 0.
If the optimization problem
%
\begin{eqnarray}
\label{LObd} &&\sum_{j=1}^n
\kappa_j \to\min!
\nonumber\\[-8pt]\\[-8pt]\nonumber
&&\quad \mbox{under the constraints } \sum_{j=1}^n
\alpha_{ij} \kappa_j \geq1, \forall1\le i\le d,
\kappa_j\ge0, \forall1\le j\le n,
\end{eqnarray}
has a solution $(\kappa_j)_{1\le j\le n}$, then
for all $\eps>0$,
\begin{eqnarray*}
P(Y_i>x~\forall1\le i\le d) & = & \mathrm{o} \bigl(x^{\eps-\sum_{j=1}^n
\kappa_j} \bigr),
\\
x^{-\eps-\sum_{j=1}^n \kappa_j} & = & \mathrm{o}\bigl(P(Y_i>x~\forall1\le i\le
d)\bigr)
\end{eqnarray*}
as $x\to\infty$.

Now suppose, in addition, that the solution is unique and that $J:=\{
j\mid\kappa_j>0\}$ has exactly $d$ elements. Then the matrix $\mathbf
{A}:=(\alpha_{ij})_{1\le i\le d, j\in J}$ is invertible. If all
entries $A^{-1}_{ji}$ of its inverse $\mathbf{A}^{-1}$ are positive, then
for all $s_i>0, 1\le i\le d$,
%
\begin{equation}
\label{asymphighdimension} \lim_{x \to\infty} \frac{P(Y_i>s_i x~\forall 1\le i\le d)}{\prod_{j\in
J}P(X_j>x^{\kappa_j})} = D \prod
_{i=1}^d s_i^{-\sum_{j\in J}A^{-1}_{ji}}
\end{equation}
with
\[
D:=\frac{1}{\llvert \det\mathbf{A}\rrvert }\frac{\prod_{k \notin J}E (
X_k^{\sum_{j \in J}\sum_{i=1}^dA_{ji}^{-1}\alpha_{ik}} )}{\prod
_{i=1}^d\sum_{j \in J}A_{ji}^{-1}}. %
\]
\end{theorem}

The assumption $A^{-1}_{ji}>0$ for all $1\le i\le d$, $j\in J$, ensures
that all factors $X_j$, $j\in J$, must be large if all products $\prod
_{j\in J} X_j^{\alpha_{ij}}$, $1\le i\le d$, are large. If\vspace*{1pt} this
condition is not satisfied, then a much more delicate analysis of the
probability that at least one of the $X_j$ is of smaller order than
$x^{\kappa_j}$ while still all products $\prod_{j\in J} X_j^{\alpha
_{ij}}$, $1\le i\le d$, exceed $x$ is needed. We conjecture that (\ref
{asymphighdimension}) holds true under much more general conditions
on $\mathbf{A}$ (in particular, that it holds for arbitrary
non-negative matrices), but such a general result is left to a future
publication.

\section{Second-order behavior of SV models with Gamma-type log-volatility}\label{sectasymptoticsSV}

The results from the two previous sections enable us to analyze the
joint extremal behavior of two lagged observations from an SV model
with Gamma-type log-volatility.

\begin{theorem}\label{limitmeasuressvmodels}
Let $(X_t, \sigma_t)_{t \in\mathbb{Z}}$ be an SV model with
Gamma-type log-volatility as in Definition~\ref{defSVnew}. Assume
that for $h \in\mathbb{N}$ there exists a unique solution to the
optimization problem
%
\begin{eqnarray}
\label{svmin} &&\sum_{i=0}^\infty
\kappa_i \to\min!
\nonumber\\[-8pt]\\[-8pt]\nonumber
&&\quad\mbox{under the constraints } \sum_{i=h}^\infty
\alpha_{i-h} \kappa_i \geq1, \sum
_{i=0}^\infty\alpha_i \kappa_i
\geq1, \kappa_i\ge0, \forall i\in\N_0.
\end{eqnarray}
%
%
In addition, suppose that $E((\eps_0^+)^{\sum_{i=0}^\infty\kappa
_i+\delta})<\infty$ for some $\delta>0$.
\begin{longlist}[(ii)]
\item[(i)] If $\kappa_i, \kappa_j >0$ for some $i \neq j$, then for all
$s_0, s_h>0$
\begin{eqnarray*}
\lim_{x \to\infty}\frac{P(\sigma_0>s_0x, \sigma_h>s_hx)}{P(\min\{
\sigma_0,\sigma_h\}>x)}&=&\lim_{x \to\infty}
\frac{P(X_0>s_0x,
X_h>s_hx)}{P(\min\{X_0,X_h\}>x)}
\\
&=&\bigl(s_0^{\alpha_i-\alpha_j}s_h^{\alpha_{j-h}-\alpha
_{i-h}}
\bigr)^{1/(\alpha_j\alpha_{i-h}-\alpha_i\alpha_{j-h})},
\end{eqnarray*}

where $\alpha_{k}:=0$ for $k<0$.
\item[(ii)] If $\kappa_i>0$ for exactly one $i \in\mathbb{N}_0$ and
$\kappa_j=0$ else, then $i\ge h$, $\alpha_i=\alpha_{i-h}$ and for
all \mbox{$s_0, s_h>0$}
\begin{eqnarray*}
&&\lim_{x \to\infty}\frac{P(\sigma_0>s_0x, \sigma
_h>s_hx)}{P(\min\{\sigma_0,\sigma_h\}>x)}
\\
&&\quad =\frac{E (\min(s_0^{-1}\prod_{j\ge h, j\ne i}
{\rm{e}}^{\alpha_{j-h}\eta_{h-j}},s_h^{-1}\prod_{j\in
\mathbb{N}_0 \setminus\{i\}} {\rm{e}}^{\alpha_j\eta_{h-j}}
)^{1/\alpha_{i}} )}{E (\min(\prod_{j\ge h,
j\ne i} {\rm{e}}^{\alpha_{j-h}\eta_{h-j}},\prod_{j\in
\mathbb{N}_0 \setminus\{i\}} {\rm{e}}^{\alpha_j\eta_{h-j}}
)^{1/\alpha_{i}} )},
\end{eqnarray*}
and
\begin{eqnarray*}
&&\lim_{x \to\infty}\frac{P(X_0>s_0x, X_h>s_hx)}{P(\min\{
X_0,X_h\}>x)}
\\
&&\quad =\frac{E (\min(s_0^{-1}\varepsilon_0^+\prod_{j\ge
h, j\ne i} {\rm{e}}^{\alpha_{j-h}\eta_{h-j}},s_h^{-1}\varepsilon
_h^+\prod_{j\in\mathbb{N}_0 \setminus\{i\}} {\rm
{e}}^{\alpha_j\eta_{h-j}} )^{1/\alpha_{i}} )}{E
(\min(\varepsilon_0^+\prod _{j\ge h, j\ne i} {\rm
{e}}^{\alpha_{j-h}\eta_{h-j}},\varepsilon_h^+\prod _{j\in
\mathbb{N}_0 \setminus\{i\}} {\rm{e}}^{\alpha_j\eta_{h-j}}
)^{1/\alpha_{i}} )}.
\end{eqnarray*}
\end{longlist}
In both cases, $(\sigma_0,\sigma_h)$ and $(X_0,X_h)$ are regularly
varying on $\E^2=(0,\infty)^2$ with index $-\sum_{i=0}^\infty\kappa_i$.
\end{theorem}

%
%
\begin{remark}
(i) According to Definition~\ref{defSVnew}(a) one has $\alpha
_k=1$ for some $k\in\N_0$. Because $\kappa_k=\kappa_{k+h}=1$ and
$\kappa_i=0$ for all other $i\in\N_0$ defines a feasible solution of
(\ref{svmin}), the optimal solution satisfies $\sum_{i=0}^\infty
\kappa_i\le2$. Hence, the moment condition on $\eps_0^+$ is always
fulfilled if $E((\eps_0^+)^{2+\delta})<\infty$ for some $\delta>0$.

(ii) Again, a result similar to Theorem~\ref{limitmeasuressvmodels}
holds for $(\llvert X_0\rrvert,\llvert X_h\rrvert )$ instead of
$(X_0,X_h)$ with $\varepsilon_i^+$
replaced by $\llvert \varepsilon_i\rrvert $.
\end{remark}

In general, the optimization problem (\ref{svmin}) must be solved
numerically. To this end, simple rules may help to find an optimal
solution $(\kappa_i)_{i\in\N_0}$ more easily. For example, if
$\kappa_i>0$, then necessarily $\alpha_i+\alpha_{i-h}\ge1$, because
otherwise $\tilde\kappa_k:=\kappa_k+\alpha_i\kappa_i, \tilde
\kappa_{k+h}:=\kappa_{k+h}+\alpha_{i-h}\kappa_i, \tilde\kappa
_i:=0$ and $\tilde\kappa_j:=\kappa_j$ for all $j\in\N_0\setminus\{
k,k+h,i\}$ with $\alpha_k=1$ defines a feasible solution with a
smaller total sum. This way, (\ref{svmin}) is reduced to a finite-dimensional program. A special case are strictly decreasing
coefficients, where the solution can be easily determined as the
following corollary shows.

\begin{corollary}\label{limitmeasuressvmodelsmonotone}
Let $(X_t,\sigma_t)_{t \in\mathbb{Z}}$ be an SV model with
Gamma-type log-volatility as in Definition~\ref{defSVnew} with
$\alpha_i, i \in\mathbb{N}_0$, strictly decreasing (which implies
$\alpha_0=1$).
Then the unique solution to~(\ref{svmin}) is given by $\kappa
_0=1-\alpha_h, \kappa_h=1$ and $\kappa_i=0$ else. Furthermore, if
$E((\eps_0^+)^{2-\alpha_h+\delta})<\infty$ for some $\delta>0$, then
\[
\lim_{x \to\infty}\frac{P(\sigma_0>s_0x, \sigma_h>s_hx)}{P(\min\{
\sigma_0,\sigma_h\}>x)} = \lim_{x \to\infty}
\frac{P(X_0>s_0x,
X_h>s_hx)}{P(\min\{X_0,X_h\}>x)}= s_0^{\alpha_h-1}s_h^{-1}
\]
for all $ s_0, s_h>0$.
\end{corollary}

Theorem~\ref{limitmeasuressvmodels} shows that under the stated
assumptions the coefficient of tail dependence is the same for the
vectors $(\sigma_0, \sigma_h)$ and $(X_0, X_h)$ and equal to $\eta
_h=1/\sum_{i=0}^\infty\kappa_i\in[1/2,1]$. In the situation of
Corollary~\ref{limitmeasuressvmodelsmonotone}, one has $\eta
_h=1/(2-\alpha_h)$. In particular, for the AR(1) model considered in
Example~\ref{exAR1} with $\alpha_h=\alpha^h, h \in\mathbb
{N}_0$, for some $\alpha\in(0,1)$, the coefficient of tail dependence
of the lagged vectors is given by $1/(2-\alpha^h)$.

If the sequence of coefficients $\alpha_h$ is decreasing, the
coefficient of tail dependence is decreasing in $h$ as well and
converges to $1/2$ as $h \to\infty$. Thus, the extremal dependence
gets weaker over time and its speed of convergence depends solely on
the values of $\alpha_h, h \in\mathbb{N}$ (resp., on $\alpha
\in(0,1)$ in the AR(1) model). The strictly monotonic decay of the
coefficients of tail dependence seems a very reasonable assumption for
asymptotically independent time series. Corollary~\ref
{limitmeasuressvmodelsmonotone} shows that SV models with Gamma-type
log-volatility allow for all possible strictly monotonically decreasing
functions $h \mapsto\eta_h \in[1/2, 1]$, provided $\sum_{h=1}^\infty
(2-1/\eta_h)<\infty$. Moreover, it is also possible to
reproduce arbitrary finite sequences of (not necessarily decreasing)
coefficients of tail dependence $\eta_h$ as long as they reflect a
non-negative dependence (i.e., stay in the interval $[1/2,1]$).

\begin{theorem}\label{representeta}
To each vector $(\eta_1, \ldots, \eta_m) \in[1/2,1]^m$ ($m \in
\mathbb{N}$) there exists an SV model with Gamma-type log-volatility
$(X_t, \sigma_t)_{t \in\mathbb{Z}}$ such that the coefficient of
tail dependence of $(\sigma_0, \sigma_h)$ and $(X_0, X_h)$ equals
$\eta_h$ for all $1\leq h \leq m$.
\end{theorem}
%

\begin{remark}\label{asymptoticdependence}
The preceding theorem shows that SV models with Gamma-type
log-volatility are also able to reflect $\eta_h=1$ for $h>0$. Remember
that asymptotic dependence of the vector $(X_0, X_h)$ implies $\eta
_h=1$ but not the other way round. In fact, it depends on the value of
$\beta$ in~(\ref{etaasymptotic}) whether our model allows for
asymptotic dependence of lagged observations.
\begin{longlist}[(ii)]
\item[(i)] If $\beta>-1$, then all vectors $(\sigma_0, \sigma_h)$ and
$(X_0, X_h)$ show asymptotic independence by Theorem~\ref
{pointprocesses} and the following conclusions.

\item[(ii)] If $\beta<-1$ and $\eta_h=1$ for some $h>0$, then the vectors
$(\sigma_0, \sigma_h)$ and $(X_0, X_h)$ show asymptotic dependence.
See Section~\ref{proofsection} for details.
\end{longlist}
\end{remark}

We conclude this section with a comparison of our work and the results
by Kulik and Soulier \cite{KuSo13}, Sections~3 and 4, who
consider a similar class of SV models. They analyze the limit distributions
%
\begin{equation}
\label{tailchain} \lim_{x\to\infty} P\bigl(\sigma_0 \leq
s_0x, \sigma_1\leq s_1x^{\rho
_1},\ldots, \sigma_h \leq s_h x^{\rho_h} \mid
\sigma_0>x\bigr)
\end{equation}
for a suitable choice of so-called ``conditional scaling exponents''
$\rho_j, j \in\mathbb{N}$, which lead to a non-degenerate limit. So
while we examine the joint extremal behavior of consecutive
volatilities or returns using regular variation on the cones $\E^d$,
Kulik and Soulier \cite{KuSo13} work in the framework of
conditional extreme value models, which are discussed, for example, in
Das and Resnick \cite{DaRe11}. In the case of asymptotic
dependence with $\rho_i=1$ for all $i \in\mathbb{N}$, the resulting
limit process is known as the tail process [cf. Basrak and Segers
\cite{BaSe09}]. Kulik and Soulier \cite{KuSo13} consider
an AR(1) model for the log-volatilities where the innovations $\eta_t,
t \in\mathbb{Z}$, have a double exponential distribution, that is, they
are symmetric with $P(\eta_t>x)=\exp(-\alpha x)/2$ for $x>0$.
This assumption fits into our model [we restrict our analysis to the
case $\alpha=1$ by standardization, cf. Remark~\ref
{remstandardization}(i)]. Furthermore, they deal with linear models
of the form (\ref{eqlogvol}) with double exponentially
distributed
innovations under the additional assumptions that $\sum_{i=1}^\infty
\alpha_i^2<\infty$ (i.e., they allow for a non-summable sequence of
coefficients) and that $\alpha_0=1, \alpha_i<1$ for all $i\geq1$.

Kulik and Soulier show that for those models a non-degenerate limit in
(\ref{tailchain}) exists if and only if the conditional scaling
exponents $\rho_j$ are chosen equal to $\alpha_j$ for all $j\in\N$.
As $\alpha_j<1$ for all $j\in\N$, this implies asymptotic
independence of consecutive volatilities and of consecutive returns.
Like Theorem~\ref{limitmeasuressvmodels}, convergence (\ref
{tailchain}) conveys refined information on their extremal dependence
structure, but the focus of the approach by Kulik and Soulier is quite
different from ours, and their mathematical techniques are in a sense
considerably simpler than the ones employed in the present paper.
Indeed, convergence (\ref{tailchain}) can be heuristically explained
by the classical ``Breiman's principle'', according to which the tail
behavior of a product is largely determined by the most heavy tailed
factor. Under the condition $\alpha_j<\alpha_0$ for all $j\in\N$, a
large value of $\sigma_0=\prod_{j=0}^\infty{\rm{e}}^{\alpha_j\eta
_{-j}}$ is most
likely caused by a large value of $\eta_0$. This in turn implies that,
given an extreme event at time 0, the lagged volatility $\sigma
_h=\prod_{j=0}^\infty{\rm{e}}^{\alpha_j\eta_{h-j}}$ will be
roughly of the order ${\rm{e}}^{\alpha_h\eta_0}$, which yields $\rho
_h=\alpha_h$ for all $h \in\mathbb{N}$.

In contrast, we consider events of the type that both $\sigma_0$ and
$\sigma_h$ exceed the same large threshold. The simple heuristic of
above fails in this setting since our results show that a single
extreme event at time 0 is not necessarily the most probable cause for
the joint exceedance. Instead one has to find combinations of two
factors ${\rm{e}}^{\eta_j}$ which are both sufficiently large (though
potentially smaller than the single factor considered in the
conditional extreme value approach) such that both products $\prod
_{j=0}^\infty{\rm{e}}^{\alpha_j\eta_{-j}}$ and $\prod_{j=0}^\infty
{\rm{e}}^{\alpha_j\eta_{h-j}}$ are large, which leads to the linear
optimization problems investigated in the Sections~\ref
{sectpowerproducts} and~\ref{sectasymptoticsSV}. This clearly shows
that in general one should neither expect a simple relationship between
the coefficients of tail dependence obtained in this paper on the one
hand and the conditional scaling exponents considered in Kulik and
Soulier \cite{KuSo13} on the other hand, nor between the
respective limiting measures arising in both approaches. This fact
somewhat qualifies the heuristic reasoning given in Section~1.5 of
Kulik and Soulier~\cite{KuSo13}.


\section{Proofs}\label{proofsection}

\subsection{Proofs to Section~\texorpdfstring{\protect\ref{sectBreiman}}{2}}

The following inequality will be useful. For all $\mathbf{x} \in
\mathbb{E}^d$ and $\mathbf{A} \in\mathbb{R}^{d \times d}$, we have
%
\begin{equation}
\label{tauinequality} \tau(\mathbf{Ax})= \tau(\mathbf{x})\tau\biggl
(\mathbf{A}
\frac{\mathbf{x}}{\tau(\mathbf{x})} \biggr) \leq\tau(\mathbf{x})\tau
(\mathbf{A}).
\end{equation}

\begin{pf*}{Proof of Lemma \ref{lemmatauprop}}
\begin{longlist}
\item[\textbf{``$\mathbf{(i) \Rightarrow(ii)}$''}:]
By definition $\tau(\mathbf{A})=0$ if and only if $\mathbf{A}\SS^d$
and $\E^d$ are disjoint, which in turn is equivalent to $\mathbf
{A}^{-1}(\E^d)\cap\E^d=\varnothing$.

Now suppose that $\tau(\mathbf{A})=\infty$. Then there exists
$\mathbf{x}\in\SS^d$ such that $\min(\mathbf{Ax})>-\min(-\mathbf
{A1})$. Thus $\mathbf{y}:=\mathbf{x}-\mathbf{1}\in\OO^d\subset(\E
^d)^c$ and $\min(\mathbf{Ay})\geq\min(\mathbf{Ax})+\min(-\mathbf
{A1})>0$, which implies \mbox{$\mathbf{Ay}\in\E^d$}, and hence $\mathbf
{A}^{-1}(\E^d)\not\subset\E^d$. Since this contradicts (i), we have
shown that (i) $\Rightarrow$ (ii).

\item[\textbf{``$\mathbf{(ii) \Rightarrow(iii)}$''}:]
Note that $\mathbf{A}$ must be invertible if $\tau(\mathbf{A})\in
(0,\infty)$. To see this, suppose that $\mathbf{A}$ were not
invertible and choose some $\mathbf{y}\ne\mathbf{0}$ satisfying
$\mathbf{A}\mathbf{y}=0$. Because $\tau(\mathbf{A})>0$, there
exists a vector $\mathbf{x}\in\E^d$ such that $\tau(\mathbf
{Ax})>0$. Moreover, for some $\lambda_0\in\R$ one has $\mathbf
{x}+\lambda_0\mathbf{y}\in\OO^d$ and $\mathbf{x}+\lambda\mathbf
{y}\in\E^d$ for all $\lambda$ between 0 and $\lambda_0$. By (\ref{tauinequality}), we have for those $\lambda$
\[
\tau(\mathbf{A})\ge\frac{\tau(\mathbf{A}(\mathbf{x}+\lambda
\mathbf{y}))}{\tau(\mathbf{x}+\lambda\mathbf{y})}=\frac{\tau
(\mathbf{Ax})}{\tau(\mathbf{x}+\lambda\mathbf{y})} \to\infty%
\]
as $\lambda\to\lambda_0$. Since this contradicts the assumption
$\tau(\mathbf{A})\in(0,\infty)$, the matrix $\mathbf{A}$ must be
invertible. In addition, if $\tau(\mathbf{A})\in(0,\infty)$ then
all entries of $\mathbf{A}^{-1}$ must be non-negative. To see this,
suppose that there exists a negative entry. Then there exists $\tilde
{\mathbf{y}} \in\mathbb{E}^d$ with $\mathbf{y}:=\mathbf
{A}^{-1}\tilde{\mathbf{y}} \in([0,\infty)^d)^c$. Furthermore, as
above there exists $\mathbf{x}\in\E^d$ such that $\mathbf{Ax}\in\E
^d$. Choose $\lambda_0>0$ such that $\mathbf{x}+\lambda_0\mathbf
{y}\in\OO^d$. Then, for all $\lambda\in[0,\lambda_0)$ we have
\[
\tau(\mathbf{A})\ge\frac{\tau(\mathbf{A}(\mathbf{x}+\lambda
\mathbf{y}))}{\tau(\mathbf{x}+\lambda\mathbf{y})}= \frac{\tau
(\mathbf{Ax}+\lambda\tilde{\mathbf{y}})}{\tau(\mathbf{x}+\lambda
\mathbf{y})}\ge\frac{\tau(\mathbf{Ax})}{\tau(\mathbf{x}+\lambda
\mathbf{y})}
\to\infty%
\]
by (\ref{tauinequality}) as $\lambda\nearrow\lambda_0$. Since this
contradicts the assumption, (iii) follows from (ii).

\item[\textbf{``$\mathbf{(iii) \Rightarrow(i)}$''}:]
If $\mathbf{A}$ is
invertible with non-negative entries of $\mathbf{A}^{-1}$, then each
row of $\mathbf{A}^{-1}$ has at least one strictly positive entry and
thus $\mathbf{A}^{-1}\mathbf{x} \in\mathbb{E}^d$ for all $\mathbf
{x} \in\mathbb{E}^d$. Therefore, (i) follows from~(iii).

The last statement of the lemma follows from (i) by continuity of the
linear mapping $\mathbf{x} \mapsto\mathbf{A}^{-1}\mathbf{x}$.
\end{longlist}
\end{pf*}

\begin{pf*}{Proof of Lemma \ref{twolemmas}}
Note first that we have $\{\mathbf{x} \in\E^d\dvtx  \tau(\mathbf{x})=y\}
=y \cdot\mathbf{1}+\mathcal{O}^d$ for all $y>0$. Now, since $0<\tau
(\mathbf{A})<\infty$ by assumption and $\tau(t\mathbf{x})=t\tau
(\mathbf{x})$ for all $\mathbf{x} \in\mathbb{E}^d$ and $t>0$, it
follows from the definition of $\tau(\mathbf{A})$ that
\begin{eqnarray*}
\label{maximizeorminimize} 1&=&\inf_{\mathbf{x} \in
\mathbb{E}^d\dvtx  \tau(\mathbf{Ax})=\tau(\mathbf{A})}\tau(\mathbf{x})
\\
&=& \inf_{\mathbf{x}\dvtx  \mathbf{Ax} \in \tau(\mathbf
{A})\mathbf{1}+\mathcal{O}^d}\tau(\mathbf{x})
\\
&=& \inf_{\mathbf{y} \in\mathcal{O}^d}\tau\bigl(\mathbf{A}^{-1}\tau(
\mathbf{A})\mathbf{1}+\mathbf{A}^{-1}\mathbf{y}\bigr)
\\
&=& \tau\bigl(\mathbf{A}^{-1}\tau(\mathbf{A})\mathbf{1}\bigr)=\tau(
\mathbf{A})\tau\bigl(\mathbf{A}^{-1}\mathbf{1}\bigr),
\end{eqnarray*}
where in the last but one equation we have used that $\mathbf
{A}^{-1}\mathbf{y} \in[0,\infty)^d$ because all entries of $\mathbf
{A}^{-1}$ are non-negative by Lemma~\ref{lemmatauprop}. Thus, $\tau
(\mathbf{A})=(\tau(\mathbf{A}^{-1}\mathbf{1}))^{-1}$ and the
statement follows from the definition of $\tau$.
\end{pf*}

\begin{pf*}{Proof of Theorem \ref{hiddenBreiman}}
For all Borel sets $B \subset\mathbb{E}^d$ bounded away from $\OO^d$
there exists a constant $\delta_B>0$ such that $\tau(\mathbf
{x})=\min(\mathbf{x})>\delta_B$ for all $\mathbf{x} \in B$. Hence,
for all $x>0$, $M>0$,
\begin{eqnarray*}
P\bigl(\mathbf{AZ} \in xB, \tau(\mathbf{A})>M\bigr) &\leq& P\bigl(\tau(
\mathbf{AZ})>x\delta_B, \tau(\mathbf{A})>M\bigr)
\\
&\leq& P\bigl(\tau(\mathbf{A})\min(\mathbf{Z})>x\delta_B, \tau(
\mathbf{A})>M\bigr)
\\
&=&P\bigl(\mathds{1}_{\{\tau(\mathbf{A})>M\}}\tau(\mathbf{A})\min
(\mathbf{Z})>x
\delta_B\bigr),
\end{eqnarray*}
where the second inequality follows from (\ref{tauinequality}) and
because $\mathbf{AZ} \in xB$ implies $\mathbf{Z} \in\mathbb{E}^d$
by Lemma~\ref{lemmatauprop}.
Since $\min(\mathbf{Z})$ is regularly varying with index $-\alpha$
and $\mathbf{A}$ and $\mathbf{Z}$ are assumed to be independent, the
univariate version of Breiman's lemma in combination with (\ref
{finitemoment}) yields
%
\begin{equation}
\label{eqnegligible} \limsup_{x \to\infty} \frac{P(\mathbf{AZ} \in xB,
\tau(\mathbf
{A})>M)}{P(\min(\mathbf{Z})>x)}\leq
\delta_B^{-\alpha}E \bigl(\tau(\mathbf{A})^\alpha
\mathds{1}_{\{\tau(\mathbf{A})>M\}} \bigr).
\end{equation}
Since, by (\ref{tauinequality}),
\[
\tau\bigl(\mathbf{A}^{-1}\mathbf{x}\bigr)\geq\frac{\tau(\mathbf{A}
\mathbf
{A}^{-1}\mathbf{x})}{\tau(\mathbf{A})}\geq
\frac{\delta_B}{\tau
(\mathbf{A})}>0
\]
a.s. for all $\mathbf{x} \in B$, the image of $B$ under $\mathbf
{A}^{-1}$ is again a.s. bounded away from $\OO^d$.
Hence,
%
\begin{eqnarray}
\label{Pratt} \lim_{x \to\infty} \frac{P(\mathbf{A}\mathbf{Z}
\in xB, \tau(\mathbf{A})\leq M)}{P(\min(\mathbf{Z})>x)} &=& \lim
_{x \to\infty} \int_{\{\tau(\mathbf{a}) \leq M\}}\frac{P(\mathbf
{Z} \in x\mathbf{a}^{-1}B)}{P(\min(\mathbf{Z})>x)}
\,\mathrm{d}P^\mathbf{A}(\mathbf{a})
\\
\nonumber
&=& \int_{\{\tau(\mathbf{a}) \leq M\}} \nu\bigl(\mathbf{a}^{-1}B
\bigr) \,\mathrm{d}P^{\mathbf{A}}(\mathbf{a})
\\
\label{mainpart}& = & E\bigl(\nu\bigl(\mathbf{A}^{-1}B\bigr)
\mathds{1}_{\{\tau
(\mathbf{A}) \leq M\}}\bigr),
\end{eqnarray}
where $P^{\mathbf{A}}$ denotes the law of $\mathbf{A}$. For the
second equation, we used that $E(\nu(\partial(\mathbf{A}^{-1}B)))=0$
(i.e., $\nu(\partial(\mathbf{A}^{-1} B))=0$ a.s.) in combination
with Pratt's lemma [cf. Pratt \cite{Pr60}], since
the integrand in~(\ref{Pratt}) is bounded by
\[
\frac{P(\mathbf{aZ} \in xB)}{P(\min(\mathbf{Z})>x)} \leq\frac
{P(\tau(\mathbf{a})\min(\mathbf{Z}) > x\delta_B)}{P(\min(\mathbf
{Z})>x)} \leq\frac{P(M\min(\mathbf{Z}) > x\delta_B)}{P(\min
(\mathbf{Z})>x)} \to
M^\alpha\delta_B^{-\alpha}.
\]
Let $M \to\infty$ in (\ref{eqnegligible}) and (\ref{mainpart}) to
obtain (\ref{theorem}) by monotone convergence.
Note that $E(\nu(\mathbf{A}^{-1}B))$ is finite since, by the
definition of $\tau$ and homogeneity of $\nu$,
\begin{eqnarray*}
E\bigl(\nu\bigl(\mathbf{A}^{-1}B\bigr)\bigr) & \leq& E\bigl(\nu\bigl(
\mathbf{A}^{-1}[\delta_B,\infty)^d
\bigr)\bigr)\leq E\bigl(\nu\bigl(\bigl[\delta_B/\tau(\mathbf{A}),\infty
\bigr)^d\bigr)\bigr)
\\
&=& \nu\bigl([\delta_B,\infty)^d\bigr) E\bigl(\tau(
\mathbf{A})^{\alpha}\bigr)< \infty.
\end{eqnarray*}\upqed
\end{pf*}

\subsection{Proofs to Section~\texorpdfstring{\protect\ref{sectSVmodels}}{3}}
\begin{pf*}{Proof of Theorem \ref{thstatsol}}
Let $\Lambda:=\{n \in\mathbb{N}_0\dvtx \alpha_n=1\}, k:=\llvert \Lambda
\rrvert $ (our
assumptions guarantee that $k<\infty$),
%
\begin{equation}
\label{betahat} \hat{\beta}:=\cases{ k\beta+k-1, &\quad$\beta>-1$,
\cr
\beta, &
\quad$\beta<-1$}
\end{equation}
and
%
\begin{equation}
\label{Khat} \hat{K}:=\cases{ \displaystyle K^k \frac{\Gamma(\beta
+1)^k}{\Gamma(k(\beta
+1))}E
\biggl(\exp\biggl(\sum_{n \notin\Lambda} \alpha_n
\eta_n \biggr) \biggr), &\quad$\beta>-1$,
\vspace*{3pt}\cr
\displaystyle k K E
\bigl(\exp(\eta_0)\bigr)^{k-1}E \biggl(\exp\biggl(\sum
_{n
\notin\Lambda} \alpha_n \eta_n
\biggr) \biggr), &\quad$\beta<-1$}
\end{equation}
[cf. equations (7.8) and (7.9) in Rootz{\'e}n \cite{Ro86}].
It follows from Lemma 7.2 in Rootz{\'e}n \cite{Ro86} that
%
\begin{equation}
\label{tailstationary} P(\log\sigma_t>z)=P \Biggl(\sum
_{i=0}^\infty\alpha_i\eta_{t-i}>z
\Biggr) \sim\hat{K}z^{\hat{\beta}} {\rm{e}}^{-z},\qquad z \to\infty, t \in
\mathbb{Z}.
\end{equation}
Since the $\varepsilon_t, t \in\mathbb{Z}$, are assumed to be
independent of the $\eta_t, t \in\mathbb{Z}$, the stationary
solution to (\ref{sigmadef}) implies the existence of a stationary
solution $(X_t, \sigma_t)_{t \in\mathbb{Z}}$. For $s>0$, one can
conclude from (\ref{tailstationary}) that
\[
\lim_{x \to\infty}\frac{P(\sigma_0>s x)}{P(\sigma_0> x)} = \lim_{x \to
\infty}
\frac{ (\log(s x) )^{\hat{\beta}}(s x)^{-1}}{ (\log
x )^{\hat{\beta}} x^{-1}} =s^{-1}, %
\]
which shows regular variation with index $-1$ of the marginal
distributions of $(\sigma_t)_{t \in\mathbb{Z}}$.

The regular variation of the marginal distributions of $\llvert
X_0\rrvert =\sigma
_0\llvert \eps_0\rrvert $ and $X_0^+=\sigma_0\eps_0^+$ follows by
Breiman's lemma,
which (under our moment assumptions on $\varepsilon_0$) gives
\[
P\bigl(\sigma_0\llvert\eps_0\rrvert>x\bigr)\sim E
\bigl(\llvert\eps_0\rrvert\bigr) P(\sigma_0>x) \quad
\mbox{and}\quad P\bigl(\sigma_0 \eps_0^+>x\bigr)\sim
E\bigl(\eps_0^+\bigr) P(\sigma_0>x).
\]
Now the tail balance assertion (\ref{tail-balanced}) and relation
(\ref{eqXtails}) between the tails of $X_0$ and $\sigma_0$ are obvious.

The last statement follows from (\ref{eqXtails}) in combination with
Theorem 7.3 in Rootz{\'e}n \cite{Ro86}.
\end{pf*}

\begin{pf*}{Proof of Theorem \ref{pointprocesses}}

\begin{longlist}
\item[\textbf{Proof of (i).}] It follows from Theorem 7.4 in
Rootz{\'e}n \cite{Ro86} that the point processes
\[
N_n^{\log\sigma}(\cdot):=\sum_{i=1}^n
\delta_{(i/n, \log\sigma
_i-\log n -\hat{\beta}\log(\log n ) -\log\hat{K})}(\cdot),\qquad n \in\mathbb{N},
\]
converge weakly to a Poisson process $N^{\log\sigma}$ on $[0,
1]\times(-\infty, \infty]$ with intensity measure $\mathrm{d}t \times{\rm
{e}}^{-x}\,\mathrm{d}x$. Now, $\exp(\cdot)$ is a continuous function such that
the pre-image of a set $B\subset(0, \infty]$ is bounded away from
$-\infty$ if $B$ is bounded away from $0$. We may thus apply
Proposition 5.5 in Resnick \cite{Re07} to derive that
$N_n^\sigma$ converges in $M_p([0,1]\times(0, \infty])$ to a Poisson
process with intensity $\mathrm{d}t \times z^{-2}\,\mathrm{d}z$ on $[0, 1] \times(0,
\infty]$ which we denote by $N^\sigma$.\vadjust{\goodbreak}

\item[\textbf{Proof of (ii).}] To
derive the second assertion, for $n \in\mathbb{N}$, introduce the
point process $N_n^{(\sigma,\varepsilon)}$ which consists of the points
$(i/n, \sigma_i/a_n,\varepsilon_i), i=1,2,\ldots,n$. The first
assertion implies that $N_n^{(\sigma,\varepsilon)}$ converges weakly to
the Poisson process $N^{(\sigma,\varepsilon)}$ with points
$(t_{(i)},z_{(i)},\varepsilon_{(i)})_{i \in\mathbb{N}}$ which has the
intensity $\mathrm{d}t \times z^{-2}\,\mathrm{d}z \times \mathrm{d}P^{\varepsilon_0}$. One may
now proceed
similarly as in the proof of Proposition 7.5 in Resnick
\cite{Re07} to show that the point processes $N_n^{X}, n \in
\mathbb{N}$, which consist of the points $(i/n, a_n^{-1}\sigma_i\cdot
\varepsilon_i), i=1,2,\dots,n$, converge to a Poisson point process
$N^{X}$ with points $(t_{(i)},z_{(i)}\cdot\varepsilon_{(i)})_{i \in
\mathbb{N}}$. Note that the additional first component $t_{(i)}$ of
the points, the dependence between the $\sigma_i, i=1,2, \dots,$ and
the possibly negative sign of the $\varepsilon_i$ do not cause
substantial
changes in the course of the proof. The derivation of the stated
intensity $\mathrm{d}t \times z^{-2} [\mathds{1}_{\{z<0\}}E(\varepsilon
_0^-)+\mathds{1}_{\{z>0\}}E(\varepsilon_0^+) ]\,\mathrm{d}z$ of the process
$N^{X}$ follows by an application of Proposition 5.2 in Resnick \cite
{Re07} to the continuous mapping
\[
T\dvtx [0,1]\times(0, \infty] \times\bigl((-\infty, \infty)\setminus\{0\}
\bigr) \to
D,\qquad (t,x,y) \mapsto(t,x\cdot y)
\]
in combination with a truncation argument like in step 4 of the proof
of Proposition 7.5 in Resnick~\cite{Re07}.\quad\qed
\end{longlist}\noqed
\end{pf*}

\subsection{Proofs to Section~\texorpdfstring{\protect\ref{sectpowerproducts}}{4}}

We start with a technical result on the tail behavior of a product of
two factors.

\begin{lemma}\label{power-products-simple}
Let $X, Y \geq0$ be two independent random variables, such that both
$X$ and $Y$ are regularly varying with index $-1$. Then, for $\alpha,
\beta$ such that $0\leq\beta<\min\{1,\alpha\}$ and all $\varepsilon
>0$ there exists an $M=M(\varepsilon)>0$ such that
%
\begin{equation}
\label{power-products-simple-2} (1-\varepsilon)\frac{\alpha}{\alpha-\beta
}s^{-1/\alpha\mp\varepsilon} \leq
\frac{P(X>x, X^\beta Y^\alpha>sx)}{P(X>x) P(Y^\alpha>x^{1-\beta
})}\leq(1+\varepsilon)\frac{\alpha}{\alpha-\beta}s^{-1/\alpha\pm
\varepsilon}
\end{equation}
for all $s,x>0$ such that $\min\{x^{1-\beta}, sx^{1-\beta}\}>M$.
\end{lemma}

\begin{pf}
If $\beta=0$, then the statement follows from the independence of $X$
and $Y$ and by applying the Potter bounds to the function $x \mapsto
P(Y^\alpha>x)$ which is regularly varying with index $-1/\alpha$.

In what follows, assume $\beta>0$. For a start, check that
%
\begin{eqnarray}\label{twosummands}
\nonumber
&& \frac{P(X>x, X^\beta Y^\alpha>sx)}{P(X>x)P
(Y^\alpha>x^{1-\beta} )}
\\
\nonumber
&&\quad = \frac{\int_0^\infty P (X>\max\{x,(sx/y)^{1/\beta
}\} ) P^{Y^\alpha}(\mathrm{d}y)}{P(X>x)P (Y^\alpha>x^{1-\beta
} )}
\nonumber\\[-8pt]\\[-8pt]\nonumber
&&\quad = \int_0^{sx^{1-\beta}} \frac{P (X>(sx/y)^{1/\beta
} )}{P(X>x)P (Y^\alpha>x^{1-\beta} )}
P^{Y^\alpha
}(\mathrm{d}y)
\\
&&\qquad {} +\int_{sx^{1-\beta}}^\infty
\frac{P (X>x )}{P(X>x)P (Y^\alpha>x^{1-\beta}
)} P^{Y^\alpha}(\mathrm{d}y).\nonumber
\end{eqnarray}
The second summand equals
%
\begin{equation}
\label{firstpotter} \frac{P(Y^\alpha>sx^{1-\beta})}{P(Y^\alpha
>x^{1-\beta})} \in\bigl[(1-\varepsilon)s^{-1/\alpha\mp\varepsilon},(1+
\varepsilon)s^{-1/\alpha\pm
\varepsilon} \bigr]
\end{equation}
for $\varepsilon>0$ if $\min\{ sx^{1-\beta},x^{1-\beta}\}>N$ for some
$N=N(\varepsilon)$ by the Potter bounds.
Again by the Potter bounds,
to each $\varepsilon>0$ there exists $N'=N'(\varepsilon)$ such that for all $x>N'$
%
\begin{eqnarray}\label{momenttail}
\nonumber
&& \int_0^{sx^{1-\beta}}
\frac{P
(X>(sx/y)^{1/\beta} )}{P(X>x)} P^{Y^\alpha}(\mathrm{d}y)
\\
&&\quad \leq (1+\varepsilon) \int
_0^{sx^{1-\beta}} \biggl(
\frac{(sx/y)^{1/\beta}}{x} \biggr)^{-1+\eps} P^{Y^\alpha}(\mathrm{d}y)
\\
&&\quad  =  (1+\varepsilon) \bigl(sx^{1-\beta}\bigr)^{-(1-\varepsilon)/\beta} \int
_0^{sx^{1-\beta}}y^{(1-\varepsilon)/\beta} P^{Y^\alpha}(\mathrm{d}y),\nonumber
\end{eqnarray}
because $(sx/y)^{1/\beta}>x$.
Since the distribution of $Y^\alpha$ is regularly varying with index
$-1/\alpha$ and $(1-\eps)/\beta>1/\alpha$ for sufficiently small
$\varepsilon>0$ by assumption, a generalization of Karamata's theorem
[Bingham  {\it et~al.} \cite{BiGoTe87}, Theorem 1.6.4] yields
\[
\lim_{t\to\infty} \frac{\int_0^t y^{(1-\varepsilon)/\beta}
P^{Y^\alpha}(\mathrm{d}y)}{t^{(1-\varepsilon)/\beta} P(Y^\alpha>t)} = \frac{1/\alpha
}{(1-\varepsilon)/\beta-1/\alpha}
\\
= \frac{\beta}{\alpha(1-\varepsilon)-\beta}. %
\]
Thus, for a suitable $N''=N''(\varepsilon)$ and $x>N''$, the integral in
(\ref{momenttail}) is bounded from above by
\[
(1+\varepsilon)\frac{\beta}{\alpha(1-\varepsilon)-\beta}\bigl(sx^{1-\beta
}\bigr)^{(1-\varepsilon)/\beta} P
\bigl(Y^\alpha>sx^{1-\beta} \bigr). %
\]
Hence, by (\ref{firstpotter}), the first summand in (\ref
{twosummands}) is bounded by
%
\begin{equation}
\label{secondpotter} (1+\varepsilon)^3 \frac{\beta}{\alpha(1-\varepsilon
)-\beta} s^{-1/\alpha\pm\varepsilon}
\end{equation}
for $x$ large enough.
It follows from (\ref{firstpotter}) and (\ref{secondpotter}) that for
given $\varepsilon>0$ one may find a suitable $\delta>0$ and a
corresponding constant $M(\delta)$ such that
\begin{eqnarray*}
\frac{P(X>x, Y^\alpha X^\beta>sx)}{P(X>x) P(Y^\alpha
>x^{1-\beta})}
&\leq& (1+\delta) s^{-1/\alpha\pm\delta}+(1+\delta)^3
\frac
{\beta}{\alpha(1-\delta)-\beta}s^{-1/\alpha\pm\delta}
\\
&\leq& (1+\varepsilon)\frac{\alpha}{\alpha-\beta}s^{-1/\alpha\pm
\varepsilon},
\end{eqnarray*}
for $\min\{sx^{1-\beta},x^{1-\beta}\}>M(\delta)$ which gives the
upper bound in (\ref{power-products-simple-2}).

Using the lower Potter bound instead of the upper one and proceeding
analogously, we arrive~at
\[
(1-\varepsilon)^3s^{-1/\alpha\mp\varepsilon}\frac{\beta}{\alpha
(1+\varepsilon)-\beta}\quad\mbox{and}\quad (1-
\varepsilon)s^{-1/\alpha\mp\varepsilon}
\]
as lower bounds for the first and second summand in (\ref
{twosummands}), respectively, which leads to the lower bound in (\ref
{power-products-simple-2}).
\end{pf}


\begin{pf*}{Proof of Proposition \ref{power-products}}
By our assumptions, the solution $(\kappa_1, \kappa_2)$ to the linear
program (\ref{linearequations}) is unique with $\kappa_1>0$ as well
as $\kappa_2>0$. Therefore, it is impossible that $(\alpha_1 \geq
\beta_1, \alpha_2 \geq\beta_2)$ or $(\alpha_1\leq\beta_1, \alpha
_2 \leq\beta_2)$, since this would imply a redundant restriction in
(\ref{linearequations}) and thus multiple solutions or $\min\{\kappa
_1,\kappa_2\}=0$. Hence, one of the points $(\alpha_i,\beta_i)$,
$i\in\{1,2\}$, must lie above or on the main diagonal and one point
below or on the main diagonal. W.l.o.g., we may assume that $(\alpha
_1,\beta_1)$ lies below or on the main diagonal, that is, $\alpha
_1\geq\beta_1$. (Otherwise, interchange
$(\alpha_1,\beta_1, X_1, \kappa_1)$ and $(\alpha_2,\beta_2, X_2,
\kappa_2)$, which leaves the assertion unchanged.)

Next, note that
$(\alpha_1 \geq\alpha_2, \beta_1 \geq\beta_2)$ or $(\alpha_1\leq
\alpha_2, \beta_1 \leq\beta_2)$ would imply that a solution with
either $\kappa_1=0$ or $\kappa_2=0$ exists.
Furthermore, $\alpha_1=\alpha_2$ or $\beta_1=\beta_2$ will
always lead to multiple solutions or $\min\{\kappa_1,\kappa_2\}=0$.
Hence, we may conclude that $\min\{\alpha_1,\beta_2\}\ge\max\{
\alpha_2,\beta_1\}$ and that $\alpha_1>\alpha_2$ and $\beta
_2>\beta_1$.

Obviously, the equations
%
\begin{equation}
\label{kappaproperty} \kappa_1\alpha_1+\kappa_2
\alpha_2=1,\qquad \kappa_1\beta_1+
\kappa_2\beta_2=1
\end{equation}
must hold, so that
\[
\kappa_1=\frac{\beta_2-\alpha_2}{\alpha_1\beta_2-\alpha_2\beta
_1}, \qquad\kappa_2=
\frac{\alpha_1-\beta_1}{\alpha_1\beta
_2-\alpha_2\beta_1}. %
\]

By the Potter bounds,
%
\begin{equation}
\label{firstfactorinterval} \frac{P(z_1X_1>x^{\kappa
_1})P(z_2X_2>x^{\kappa_2})}{P(X_1>x^{\kappa
_1})P(X_2>x^{\kappa_2})} \in\bigl[(1-\varepsilon)z_1^{1\mp\varepsilon
}z_2^{1\mp\varepsilon},
(1+\varepsilon)z_1^{1\pm\varepsilon}z_2^{1\pm\varepsilon} \bigr]
\end{equation}
for all $\varepsilon>0$ provided $\min\{x^{\kappa_1}, z_1^{-1}x^{\kappa
_1}, x^{\kappa_2}, z_2^{-1}x^{\kappa_2}\}>N(\varepsilon)$ for
sufficiently large $N(\eps)$.

Set for abbreviation $\tilde{X}_i:=z_iX_i$, $i=1,2$, and note that
%
\begin{eqnarray}\label{foursummands}
&& \frac{P((z_1X_1)^{\alpha_1} (z_2X_2)^{\alpha
_2}>x,(z_1X_1)^{\beta_1} (z_2X_2)^{\beta_2} >x)}{P(z_1X_1>x^{\kappa
_1})P(z_2X_2>x^{\kappa_2})}
\nonumber
\\
\nonumber
&&\quad  =  \frac{P(\tilde{X}_1^{\alpha_1} \tilde{X}_2^{\alpha
_2}>x,\tilde{X}_1^{\beta_1} \tilde{X}_2^{\beta_2} >x, \tilde
{X}_1\leq x^{\kappa_1},\tilde{X}_2 \leq x^{\kappa_2})}{P(\tilde
{X}_1>x^{\kappa_1})P(\tilde{X}_2>x^{\kappa_2})}
\\
&&\qquad{} - \frac{P(\tilde{X}_1^{\alpha_1} \tilde
{X}_2^{\alpha_2}>x,\tilde{X}_1^{\beta_1} \tilde{X}_2^{\beta_2} >x,
\tilde{X}_1> x^{\kappa_1}, \tilde{X}_2> x^{\kappa_2})}{P(\tilde
{X}_1>x^{\kappa_1})P(\tilde{X}_2>x^{\kappa_2})}
\\
&&\qquad {}+ \frac{P(\tilde{X}_1^{\alpha_1} \tilde
{X}_2^{\alpha_2}>x,\tilde{X}_1^{\beta_1} \tilde{X}_2^{\beta_2} >x,
\tilde{X}_1> x^{\kappa_1})}{P(\tilde{X}_1>x^{\kappa_1})P(\tilde
{X}_2>x^{\kappa_2})}\nonumber
\\
&&\qquad{} + \frac{P(\tilde{X}_1^{\alpha_1} \tilde
{X}_2^{\alpha_2}>x,\tilde{X}_1^{\beta_1} \tilde{X}_2^{\beta_2} >x,
\tilde{X}_2> x^{\kappa_2})}{P(\tilde{X}_1>x^{\kappa_1})P(\tilde
{X}_2>x^{\kappa_2})}.\nonumber
\end{eqnarray}
For $x\geq1$ the first and second summand equal 0 and $-1$,
respectively, because of (\ref{kappaproperty}).

Next, note that, again by (\ref{kappaproperty}),
\begin{eqnarray*}
&& P\bigl(\tilde{X}_1^{\alpha_1}\tilde{X}_2^{\alpha_2}>x,
\tilde{X}_1^{\beta_1} \tilde{X}_2^{\beta_2}
>x, \tilde{X}_1> x^{\kappa
_1}\bigr)
\\
&&\quad = P \biggl( \biggl(\frac{\tilde{X}_1}{x^{\kappa_1}} \biggr)^{\alpha_1}
\tilde{X}_2^{\alpha_2}> x^{1-\alpha_1\kappa_1}, \biggl(
\frac{\tilde
{X}_1}{x^{\kappa_1}} \biggr)^{\beta_1}\tilde{X}_2^{\beta_2} >
x^{1-\beta_1\kappa_1},\tilde{X}_1> x^{\kappa_1} \biggr)
\\
&&\quad = P \biggl( \biggl(\frac{\tilde{X}_1}{x^{\kappa_1}} \biggr)^{\alpha_1}
\tilde{X}_2^{\alpha_2}> x^{\alpha_2\kappa_2}, \biggl(
\frac{\tilde
{X}_1}{x^{\kappa_1}} \biggr)^{\beta_1}\tilde{X}_2^{\beta_2}
>x^{\beta
_2\kappa_2}, \tilde{X}_1> x^{\kappa_1} \biggr)
\\
&&\quad = P \biggl( \biggl(\frac{\tilde{X}_1}{x^{\kappa_1}} \biggr)^{\alpha
_1/\alpha_2}
\tilde{X}_2>x^{\kappa_2}, \biggl(\frac{\tilde{X}_1}{x^{\kappa_1}}
\biggr)^{\beta_1/\beta
_2}\tilde{X}_2>x^{\kappa_2},
\tilde{X}_1> x^{\kappa_1} \biggr),
\end{eqnarray*}
where $ (\tilde{X}_1/x^{\kappa_1} )^{\alpha_1/\alpha
_2}:=\infty$ if $\alpha_2=0$.
According to the above discussion, we have $\beta_1/\beta_2<1<\alpha
_1/\alpha_2$. Therefore, the last probability equals
\begin{eqnarray*}
P \biggl( \biggl(\frac{\tilde{X}_1}{x^{\kappa_1}} \biggr)^{\beta_1/\beta
_2}\tilde{X}_2>x^{\kappa_2},
\tilde{X_1}> x^{\kappa_1} \biggr) &=& P \bigl(
\tilde{X}_1^{\beta_1/\beta_2}\tilde{X}_2>x^{\kappa
_2+\kappa_1\beta_1/\beta_2},
\tilde{X}_1> x^{\kappa_1} \bigr)
\\
& = & P \bigl(\tilde{X}_1^{\beta_1\kappa_1}\tilde{X}_2^{\beta
_2\kappa_1}>x^{\kappa_1},
\tilde{X}_1> x^{\kappa_1} \bigr).
\end{eqnarray*}
Now, we substitute $u$ for $x^{\kappa_1}/z_1$, use (\ref
{kappaproperty}) and apply Lemma~\ref{power-products-simple} and the
Potter bounds to $x\mapsto P(X_2^{\beta_2 \kappa_1}>x)$:
\begin{eqnarray*}
&& \frac{P (\tilde{X}_1> x^{\kappa_1}, \tilde
{X}_1^{\beta_1\kappa_1}\tilde{X}_2^{\beta_2\kappa_1}>x^{\kappa
_1} )}{P(\tilde{X}_1> x^{\kappa_1}) P(\tilde{X}_2>x^{\kappa
_2})}
\\
&&\quad =\frac{P (z_1X_1> x^{\kappa_1}, (z_1X_1)^{\beta_1\kappa
_1}(z_2X_2)^{\beta_2\kappa_1}>x^{\kappa_1} )}{P(z_1X_1>
x^{\kappa_1}) P(z_2X_2>x^{\kappa_2})}
\\
&&\quad =\frac{P (X_1>u, X_1^{\beta_1\kappa_1}X_2^{\beta_2\kappa
_1}>z_1^{\beta_2\kappa_2}z_2^{-\beta_2\kappa_1}u )}{P(X_1>u)P(
X_2^{\beta_2\kappa_1}>u^{1-\beta_1\kappa_1})}\cdot\frac{P(
X_2^{\beta_2\kappa_1}>u^{1-\beta_1\kappa_1})}{P( X_2^{\beta
_2\kappa_1}>z_1^{\beta_2\kappa_2}z_2^{-\beta_2\kappa_1}u^{1-\beta
_1\kappa_1})}
\\
&&\quad \in \biggl[(1-\eta)^2\frac{\beta_2}{\beta_2-\beta_1} \biggl(\frac
{z_1^{\beta_2\kappa_2}}{z_2^{\beta_2\kappa_1}}
\biggr)^{\mp2\eta},(1+\eta)^2\frac{\beta_2}{\beta_2-\beta_1} \biggl(
\frac{z_1^{\beta_2\kappa_2}}{z_2^{\beta_2\kappa_1}} \biggr)^{\pm2\eta}
\biggr]
\end{eqnarray*}
if both $u^{1-\beta_1\kappa_1}=(x^{\kappa_1}/z_1)^{\beta_2 \kappa
_2}$ and $z_1^{\beta_2\kappa_2}z_2^{-\beta_2\kappa_1}u^{1-\beta
_1\kappa_1}=(x^{\kappa_2}/z_2)^{\beta_2\kappa_1}$ are larger than
some $N(\eta)$, which is the case if both $x^{\kappa_1}/z_1>M(\eta)$
and $x^{\kappa_2}/z_2>M(\eta)$ for a suitably chosen constant $M(\eta
)$. Adapting the value of $\eta$ to each $\varepsilon>0$, we may hence
find an $M'(\varepsilon)$ such that
%
\begin{eqnarray}\label{thirdsummandinterval}
\nonumber
&& \frac{P(\tilde{X}_1^{\alpha_1} \tilde
{X}_2^{\alpha_2}>x,\tilde{X}_1^{\beta_1} \tilde{X}_2^{\beta_2} >x,
\tilde{X}_1> x^{\kappa_1})}{P(\tilde{X}_1>x^{\kappa_1})P(\tilde
{X}_2>x^{\kappa_2})}
\nonumber\\[-8pt]\\[-8pt]\nonumber
&&\quad \in\biggl[(1-\varepsilon)\frac{\beta
_2}{\beta_2-\beta_1}z_1^{\mp\varepsilon}
z_2^{\mp\varepsilon},(1+\varepsilon)\frac{\beta_2}{\beta_2-\beta
_1}z_1^{\pm\varepsilon}z_2^{\pm\varepsilon}
\biggr]
\end{eqnarray}
if both $x^{\kappa_1}/z_1>M'(\varepsilon)$ and $x^{\kappa
_2}/z_2>M'(\varepsilon)$.

Analogously, one shows for the fourth summand in equation (\ref
{foursummands}) that
%
\begin{eqnarray}
\nonumber
&& \frac{P(\tilde{X}_1^{\alpha_1} \tilde
{X}_2^{\alpha_2}>x,\tilde{X}_1^{\beta_1} \tilde{X}_2^{\beta_2} >x,
\tilde{X}_2> x^{\kappa_2})}{P(\tilde{X}_1>x^{\kappa_1})P(\tilde
{X}_2>x^{\kappa_2})}
\nonumber\\[-8pt]\\[-8pt]\nonumber
&&\quad \in\label{fourthsummandinterval} \biggl[(1-\varepsilon)\frac{\alpha
_1}{\alpha_1-\alpha_2}z_1^{\mp\varepsilon}z_2^{\mp\varepsilon},
(1+\varepsilon)\frac{\alpha_1}{\alpha_1-\alpha_2}z_1^{\pm\varepsilon
}z_2^{\pm\varepsilon}
\biggr]
\end{eqnarray}
if both $x^{\kappa_1}/z_1>M''(\varepsilon)$ and $x^{\kappa
_2}/z_2>M''(\varepsilon)$ for a suitably chosen $M''(\varepsilon)$.

Finally, for $\varepsilon>0$, combining (\ref
{firstfactorinterval})--(\ref{fourthsummandinterval}) we arrive at
\begin{eqnarray*}
&& \frac{P((z_1X_1)^{\alpha_1} (z_2X_2)^{\alpha
_2}>x,(z_1X_1)^{\beta_1} (z_2X_2)^{\beta_2}>x)}{P(X_1>x^{\kappa
_1})P(X_2>x^{\kappa_2})}
\\
&&\quad \in \bigl[\bigl(1-\varepsilon'\bigr)Cz_1^{1\mp\varepsilon'}z_2^{1\mp\varepsilon
'},
\bigl(1+\varepsilon'\bigr)Cz_1^{1\pm\varepsilon'}z_2^{1\pm\varepsilon'}
\bigr]
\end{eqnarray*}
for a suitable choice of $\varepsilon'$ and a constant $N=N(\varepsilon)$
such that $\min\{x^{\kappa_1}, z_1^{-1}x^{\kappa_1}$,
$x^{\kappa_2}, z_2^{-1}x^{\kappa_2}\}>N$.
\end{pf*}

\begin{pf*}{Proof of Theorem \ref{power-products-infinite}}
\mbox{}

\textbf{1. Almost sure convergence and joint positivity of infinite products}

The almost sure convergence of $\prod_{i=1}^k X_i^{\alpha_i}$ (in $\R
$) as $k\to\infty$ is equivalent to the almost sure convergence of
$\sum_{i=1}^k \alpha_i (\log X_i)^+$ and the latter follows because
the series is non-decreasing and $E(\sum_{i=1}^\infty\alpha_i (\log
(X_i))^+)= \sum_{i=1}^\infty\alpha_i E(\log X_1)^+<\infty$ by our
assumptions. The almost sure convergence of $\prod_{i=1}^k X_i^{\beta
_i}$ follows analogously.

Since we have assumed that $P(Y_i>0)>0$ for $i=0,1$, there is a
positive probability for the event $\{X_i>0$ for all $i$
with $ \alpha_i+\beta_i>0\}$. Thus, the probability
\[
P \Biggl(\prod_{i=1}^\infty
X_i^{\alpha_i+\beta_i}>0 \Big| X_i>0 \mbox{ for all } i
\mbox{ with } \alpha_i+\beta_i>0 \Biggr), %
\]
which can only be 0 or 1 by the 0--1-law, must be equal to 1. Hence,
$P(Y_0Y_1>0) =P(Y_0>0, Y_1>0)>0$ as well.\vadjust{\goodbreak}

\textbf{2. Reduction to a finite-dimensional linear program}

In the following, we show that the infinite-dimensional linear program
(\ref{LOb}) is equivalent to the finite-dimensional one
%
\begin{eqnarray}\label{LObfinite}
&& \sum_{i=1}^{n^\ast}
\kappa_i \to\min!
\nonumber\\[-8pt]\\[-8pt]\nonumber
&&\quad \mbox{under the constraints } \sum_{i=1}^{n^\ast}
\alpha_i \kappa_i \geq1, \sum
_{i=1}^{n^\ast} \beta_i \kappa_i
\geq1, \kappa_i\ge0, \forall1 \leq i \leq{n^\ast},
\end{eqnarray}
for a suitably chosen ${n^\ast}$. To this end, choose $i^\ast, j^\ast
\in\mathbb{N}_0$ such that $\alpha_{i^*}=\sup_{i\in\N} \alpha_i$
and $\beta_{j^*}=\sup_{j\in\N} \beta_j$. Choose $n^\ast$ such
that $\max\{\alpha_i,\beta_i\}< 1/(2/\alpha_{i^*}+2/\beta_{j^*})$
for all $i>n^\ast$. Let $(\kappa_i)_{i\in\N}$ be any feasible
solution to (\ref{LOb}) and suppose that $\kappa_i>0$ for some
$i>n^\ast$. Then
$\tilde\kappa_{i^*}:=\kappa_{i^*}+\kappa_i\alpha_i/\alpha_{i^*}$,
$\tilde\kappa_{j^*}:=\kappa_{j^*}+\kappa_i\beta_i/\beta_{j^*}$
(resp., $\tilde\kappa_{i^*}:=\kappa_{i^*}+\kappa_i\max\{\alpha_i/\alpha
_{i^*},\beta_i/\beta_{j^*}\}$ if $i^*=j^*$), $\tilde\kappa_i:=0$
and $\tilde\kappa_j:=\kappa_j$ for all $j\in\N\setminus\{
i^*,j^*,i\}$ defines a feasible solution to the constraints (\ref
{LOb}) with $\sum_{i=1}^\infty\tilde\kappa_i<\sum_{i=1}^\infty
\kappa_i$. Hence, all optimal solutions $(\kappa_i)_{i\in\N}$ to
(\ref{LOb}) satisfy $\kappa_i=0$ for all $i>n^\ast$. They can thus
be identified with optimal solutions to the finite-dimensional problem
(\ref{LObfinite}), and vice versa. It will sometimes be useful in the
following that by our definition of $n^\ast$ we have
%
\begin{equation}
\label{Eqsmallcoeffs}\max\{\alpha_i, \beta_i\} <
\frac{1}{2 (\sklfrac{1}{\alpha_{i^*}}+\sklfrac{1}{\beta_{j^*}}
)} \leq\frac{1}{2 \sum_{j=1}^\infty\kappa_j}\qquad\mbox{for all } i >
n^\ast
\end{equation}
for any optimal solution $(\kappa_j)_{j \in\mathbb{N}}$, because
$\bar\kappa_{i^*}:=1/\alpha_{i^*}$, $\bar\kappa_{j^*}:=1/\beta
_{j^*}$ and $\bar\kappa_j=0$ for all $j\in\mathbb{N} \setminus\{
i^*,j^*\}$ (if $i^*\ne j^*$) defines a feasible solution.

A unique solution to (\ref{LObfinite}) must be a basic feasible
solution [Sierksma \cite{Si96}, Theorems 1.2 and 1.5],
that is, $\kappa_l=0$ for all but at most two indices $l\in\{1,\ldots
,n^\ast\}$. Thus, for a unique solution we may assume that $\kappa
_1>0$ and $\kappa_i=0$ for all $i\ge3$ w.l.o.g.

\textbf{3. A helpful moment bound}

Let $m \in\mathbb{N}, \varepsilon\in(0,1)$ and $c_i \in[0,1-\varepsilon
), i\geq m$, such that $\sum_{i=m}^\infty c_i<\infty$. Then
%
\begin{eqnarray}\label{Eqhelpfulbound}
\nonumber
E \Biggl(\prod_{i=m}^\infty
X_i^{c_i} \Biggr)&\leq& \liminf_{n \to\infty}
\prod_{i=m}^n E \bigl(X_1^{c_i}
\bigr)
\\
&\le& \liminf_{n \to\infty} \prod
_{i=m}^n \bigl(E \bigl(X_1^{1-\varepsilon}
\bigr) \bigr)^{c_i/(1-\varepsilon)}
\\
\nonumber& = & \bigl(E\bigl(X_1^{1-\varepsilon}\bigr)
\bigr)^{\sum_{i=m}^\infty c_i/(1-\varepsilon)}<\infty
\end{eqnarray}
by Fatou's lemma and Jensen's inequality for independent and
identically distributed $X_i\geq0$ which are all regularly varying
with index $-1$.
[Note that this is the only instance in the proof where we use the
assumption that all $X_i$ are identically distributed, cf. Remark~\ref
{remconvconds}(iii).]\vadjust{\goodbreak}

\textbf{4. Proof of (i)}

We prove the assertion by formalizing the heuristic arguments used to
motivate the linear optimization problem. Let $(\kappa_i)_{i \in
\mathbb{N}}$ be an optimal (not necessarily unique) solution to (\ref
{LOb}) and note that by part 2 of this proof we have $\kappa_i=0$ for
$i>n^\ast$. By our assumptions, there exists a $\delta>0$ such that
$P(\prod_{i=n^\ast+1}^\infty X_i^{\alpha_i}>\delta, \prod_{i=n^\ast
+1}^\infty X_i^{\beta_i}>\delta)=:c>0$, cf. part 1 of
this proof. Fix $\varepsilon>0$. The lower bound (\ref{Eqlowerobound})
follows from
\begin{eqnarray*}
&& P(Y_0>x,Y_1>x)
\\
&&\quad \ge P \Biggl(\prod_{i=n^\ast+1}^\infty
X_i^{\alpha_i}>\delta, \prod_{i=n^\ast+1}^\infty
X_i^{\beta_i}>\delta\Biggr) P\bigl(X_i>\bigl(
\delta^{-1}x\bigr)^{\kappa_i}\mbox{ for all } 1 \leq i \leq
n^\ast\bigr)
\\
&&\quad \ge c \prod_{i=1}^{n^\ast} \bigl(
\delta^{-1} x\bigr)^{-\kappa_i-\eps
/n^{\ast}}= c \delta^{\sum_{i=1}^{n^\ast}\kappa_i+\eps}
x^{-\sum
_{i=1}^{n^\ast}\kappa_i-\eps}=c \delta^{\sum_{i=1}^\infty\kappa
_i+\eps} x^{-\sum_{i=1}^\infty\kappa_i-\eps}
\end{eqnarray*}
for sufficiently large $x$ because of the regular variation of the
independent random variables $X_i$.\vspace*{1pt}

To establish (\ref{Equpperobound}), write $Z:=\max\{\prod_{i=n^\ast
+1}^\infty X_i^{\alpha_i},\prod_{i=n^\ast+1}^\infty X_i^{\beta_i}\}
$ and fix $0<\varepsilon<\sum_{i=1}^\infty\kappa_i$. Note that
%
\begin{eqnarray}\label{EqIntuppbound}
\nonumber
&& P(Y_0>x,Y_1>x) x^{\sum_{i=1}^\infty\kappa_i-\eps}
\nonumber\\[-8pt]\\[-8pt]\nonumber
&&\quad \le \int P \Biggl(\prod_{i=1}^{n^\ast}
X_i^{\alpha_i}>\frac{x}z,\prod
_{i=1}^{n^\ast} X_i^{\beta_i}>
\frac{x}z \Biggr) \biggl(\frac{x}z \biggr)^{\sum_{i=1}^{n^\ast}\kappa
_i-\eps}
z^{\sum_{i=1}^{n^\ast}\kappa_i-\eps} P^Z(\mathrm{d}z).
\end{eqnarray}
Because $(\kappa_i)_{1 \leq i \leq n^\ast}$ is an optimal solution to
(\ref{LObfinite}) we have for all $y \geq1$
\begin{eqnarray*}
&& P \Biggl(\prod_{i=1}^{n^\ast}
X_i^{\alpha_i}>y,\prod_{i=1}^{n^\ast}
X_i^{\beta_i}>y \Biggr)
\\
&&\quad  \le P \Biggl(\sum
_{i=1}^{n^\ast} \alpha_i \frac{(\log X_i)^+}{\log y}
\ge1, \sum_{i=1}^{n^\ast} \beta_i
\frac
{(\log X_i)^+}{\log y}\ge1 \Biggr)
\\
&&\quad  \le P \Biggl(\sum_{i=1}^{n^\ast}
\frac{(\log X_i)^+}{\log y} \geq\sum_{i=1}^{n^\ast}
\kappa_i \Biggr)
\\
&&\quad = P \Biggl(\prod_{i=1}^{n^\ast}
\max(X_i,1)>y^{\sum_{i=1}^{n^\ast
}\kappa_i} \Biggr).
\end{eqnarray*}
This expression is a regularly varying function of $y$ with index
$-\sum_{i=1}^{n^\ast}\kappa_i$ by the corollary to Theorem 3 in
Embrechts and Goldie \cite{EmGo80}. Thus, the integrand in (\ref
{EqIntuppbound}) tends to 0 for all $z>0$ as $x \to\infty$.
Furthermore, we have
\[
E \bigl(Z^{\sum_{i=1}^{n^\ast} \kappa_i - \varepsilon} \bigr) \leq E \Biggl
(\prod
_{i=n^\ast+1}^\infty X_i^{\alpha_i(\sum
_{i=1}^{n^\ast}\kappa_i-\varepsilon)}
\Biggr)+E \Biggl(\prod_{i=n^\ast
+1}^\infty
X_i^{\beta_i(\sum_{i=1}^{n^\ast}\kappa_i-\varepsilon
)} \Biggr)<\infty%
\]
by (\ref{Eqsmallcoeffs}) and (\ref{Eqhelpfulbound}). Thus,
$Mz^{\sum_{i=1}^{n^\ast}\kappa_i-\eps}$ is an integrable majorant
to the integrand in (\ref{EqIntuppbound}) for sufficiently large
$M>0$. Hence, $P(Y_0>x,Y_1>x)=\mathrm{o}(x^{\eps-\sum_{i=1}^\infty\kappa
_i})$ follows by dominated convergence.

\textbf{5. Proof of (ii)}

In the following, we assume a unique solution to the optimization
problem (\ref{LOb}) with \mbox{$\kappa_i>0, \kappa_j>0$}. W.l.o.g. let us
assume that $i=1, j=2$, thus $\kappa_1>0, \kappa_2>0$ and $\kappa
_l=0$, $l \geq3$. Moreover, w.l.o.g. we may assume $\alpha_1>\alpha
_2$, $\beta_2>\beta_1$ and $\min\{\alpha_1,\beta_2\}\ge\max\{
\alpha_2,\beta_1\}$ (cf. the proof of Proposition~\ref{power-products}).

As in the proof of Proposition~\ref{power-products}, equation (\ref
{kappaproperty}) holds and it follows that
\[
\kappa_1=\frac{\beta_2-\alpha_2}{\alpha_1\beta_2-\alpha_2\beta
_1},\qquad \kappa_2=
\frac{\alpha_1-\beta_1}{\alpha_1\beta
_2-\alpha_2\beta_1}.
\]

There exists a so-called dual problem to (\ref{LObfinite}) [see
Sierksma \cite{Si96}, Chapter~2] that is given by
%
\begin{eqnarray}
\label{DLO}
&&\hat{\kappa}_1+\hat{\kappa}_2 \to\max!
\nonumber\\[-8pt]\\[-8pt]\nonumber
&&\quad \mbox{under the constraints } \alpha_l \hat{
\kappa}_1 + \beta_l \hat{\kappa}_2\leq1
\mbox{ for all } l=1, \dots, n^\ast, \hat{\kappa}_1 \geq0,
\hat{\kappa}_2\geq0.
\end{eqnarray}
Since we have assumed an optimal solution to the primal problem (\ref
{LObfinite}), there also exists an optimal solution to the dual problem
[Sierksma \cite{Si96}, Theorem 2.2], denoted by $(\hat
{\kappa}_1,\hat{\kappa}_2)$ with
%
\begin{equation}
\label{Eqsameobjective} \hat{\kappa}_1+\hat{\kappa}_2=
\kappa_1+\kappa_2.
\end{equation}
Because of $\kappa_1, \kappa_2>0$, it follows by the complementary
slackness theorem [Sierksma \cite{Si96}, Theorem 2.4] that
\[
\alpha_1\hat{\kappa}_1 + \beta_1\hat{
\kappa}_2=1\quad\mbox{and}\quad \alpha_2\hat{\kappa}_1
+ \beta_2\hat{\kappa}_2=1.
\]
Therefore,
%
\begin{equation}
\label{eqkappahat} \hat{\kappa}_1=\frac{\beta_2-\beta_1}{\alpha_1\beta
_2-\alpha
_2\beta_1}>0,\qquad \hat{
\kappa}_2=\frac{\alpha_1-\alpha
_2}{\alpha_1\beta_2-\alpha_2\beta_1}>0.
\end{equation}
If equality held in the constraints of (\ref{DLO}) for some $l \in\{
3, \ldots, n^\ast\}$, then direct calculations show that
\[
\tilde\kappa_1=\frac{\beta_l-\alpha_l}{\alpha_1\beta_l-\alpha
_l\beta_1}, \qquad\tilde
\kappa_l=\frac{\alpha_1-\beta_1}{\alpha
_1\beta_l-\alpha_l\beta_1}, \qquad\tilde\kappa_i=0,
\forall i\in\bigl\{2,\ldots,n^\ast\bigr\}\setminus\{l\}
\]
would be another optimal solution to the primal problem in
contradiction to our assumptions.
Therefore, there exists an $\varepsilon>0$ such that
%
\begin{equation}
\label{strictlylessthanone} \alpha_l \hat{\kappa}_1 +
\beta_l \hat{\kappa}_2<1-2\varepsilon\qquad\mbox{for all }l \in
\bigl\{3, \ldots,n^\ast\bigr\}.
\end{equation}
Let
\begin{eqnarray*}
Z_1(x_3, x_4, \dots) &:=& \prod
_{i=3}^\infty x_i^{\vafrac{\alpha_i\beta
_2-\beta_i\alpha_2}{\alpha_1\beta_2-\alpha_2\beta_1}},
\\
Z_2(x_3, x_4, \dots)&:=&\prod
_{i=3}^\infty x_i^{\vafrac{\beta_i\alpha
_1-\alpha_i\beta_1}{\alpha_1\beta_2-\alpha_2\beta_1}}
\end{eqnarray*}
if $x_i>0$ for all $i\geq3$ with $\max\{\alpha_i,\beta_i\}>0$ and
set $Z_1(x_3, x_4, \dots)=$ $Z_2(x_3, x_4, \dots)=\infty$ else.
These two quantities will play the role of $z_1$ and $z_2$ in the
statement and proof of Proposition~\ref{power-products}. For $\delta
>0$, define $M=M(\delta)$ as in Proposition~\ref{power-products},
\[
\mathbb{M}=\mathbb{M}(x,M):= \bigl\{ \mathbf{x} \in[0, \infty
)^\mathbb{N}\dvtx  M Z_1(\mathbf{x})<x^{\kappa_1}, M
Z_2(\mathbf{x})<x^{\kappa_2} \bigr\} %
\]
and $\mathbf{X}_3^\infty:=(X_3, X_4, \dots)$.

We start with showing (ii) for $s_0=s_1=1$.
%
\begin{eqnarray} \label{non-degenerate}
\nonumber
\hspace*{-30pt}&& \frac{P (\prod_{i=1}^\infty X_i^{\alpha
_i}>x, \prod_{i=1}^\infty X_i^{\beta_i}>x )}{P(X_1>x^{\kappa
_1})P(X_2>x^{\kappa_2})}
\\
\hspace*{-30pt}&&\quad =  \int_{\mathbb{M}} \frac{P ((Z_1(\mathbf{x})X_1)^{\alpha
_1}(Z_2(\mathbf
{x})X_2)^{\alpha_2}>x, (Z_1(\mathbf{x})X_1)^{\beta_1}(Z_2(\mathbf
{x})X_2)^{\beta_2}>x )}{P(X_1>x^{\kappa_1})
P(X_2>x^{\kappa_2})}
P^{\mathbf{X}_3^\infty}(\mathrm{d}\mathbf{x})
\\
\hspace*{-30pt}&&\qquad{} + \frac{P (\prod_{i=1}^\infty X_i^{\alpha_i}>x, \prod_{i=1}^\infty
X_i^{\beta_i}>x, \mathbf{X}_3^\infty\in\mathbb{M}^c
)}{P(X_1>x^{\kappa_1})P(X_2>x^{\kappa_2})}.\nonumber
\end{eqnarray}

According to Proposition~\ref{power-products} and (\ref
{eqkappahat}), the integrand of the first summand on the right-hand
side of (\ref{non-degenerate}) is
bounded from above by
\[
(1+\delta)C Z_1(\mathbf{x})^{1\pm\delta}Z_2(
\mathbf{x})^{1\pm
\delta} \leq(1+\delta) C \prod_{i=3}^\infty
x_i^{(1\pm\delta
)(\alpha_i\hat\kappa_1+\beta_i\hat\kappa_2)}
\]
for $x>M$ with $C$ as in Proposition~\ref{power-products}. Thus, if we
chose $\delta$ small enough then all the exponents in the last
expression are less than $1-\varepsilon$ (by (\ref{strictlylessthanone})
for $3 \leq i \leq n^\ast$ and by (\ref{Eqsmallcoeffs}) and (\ref
{Eqsameobjective}) for $i>n^\ast$) and the upper bound is integrable
w.r.t. $P^{\mathbf{X}_3^\infty}$ by (\ref{Eqhelpfulbound}).
Therefore, by dominated convergence, we see that
%
\begin{eqnarray}\label{finalconstant}
\nonumber
\hspace*{-10pt}&& \lim_{x \to\infty}\int_{\mathbb{M}}
\frac{P ((Z_1(\mathbf{x})X_1)^{\alpha_1}(Z_2(\mathbf
{x})X_2)^{\alpha_2}>x, (Z_1(\mathbf{x})X_1)^{\beta_1}(Z_2(\mathbf
{x})X_2)^{\beta_2}>x )}{P(X_1>x^{\kappa_1})
P(X_2>x^{\kappa_2})} P^{\mathbf{X}_3^\infty}(\mathrm{d}\mathbf{x})
\\
\nonumber
\hspace*{-10pt}&&\quad =\int_{\mathbb{M}} \lim_{x \to\infty}
\frac{P ((Z_1(\mathbf{x})X_1)^{\alpha_1}(Z_2(\mathbf
{x})X_2)^{\alpha_2}>x, (Z_1(\mathbf{x})X_1)^{\beta_1}(Z_2(\mathbf
{x})X_2)^{\beta_2}>x )}{P(X_1>x^{\kappa_1})
P(X_2>x^{\kappa_2})} P^{\mathbf{X}_3^\infty}(\mathrm{d}\mathbf{x})
\nonumber\\[-8pt] \hspace*{-10pt}\\[-8pt]\nonumber
\hspace*{-10pt}&&\quad = \int\prod_{i=3}^\infty
C Z_1(\mathbf{x})Z_2(\mathbf{x})\mathds
{1}_{\{\max\{Z_1(\mathbf{x}),Z_2(\mathbf{x})\}<\infty\}} P^{\mathbf
{X}_3^\infty}(\mathrm{d}\mathbf{x})
\\
\nonumber \hspace*{-10pt}&&\quad  =  C E \Biggl(\prod_{i=3}^\infty
X_i^{\alpha
_i\hat{\kappa}_1+\beta_i\hat{\kappa}_2} \Biggr) = D \in(0, \infty),
\end{eqnarray}
where in the last step we have used that $\hat{\kappa}_1, \hat
{\kappa}_2>0$ and thus the product vanishes if $Z_1(\mathbf{x})$ or
$Z_2(\mathbf{x})$ is infinite.

It remains to be shown that
%
\begin{eqnarray}\label{eqremainder}
\nonumber
&& \frac{P (\prod_{i=1}^\infty X_i^{\alpha
_i}>x, \prod_{i=1}^\infty X_i^{\beta_i}>x, \mathbf{X}_3^\infty\in
\mathbb{M}^c )}{P(X_1>x^{\kappa_1})P(X_2>x^{\kappa_2})}
\\
&&\quad \le \frac{P (\prod_{i=1}^\infty X_i^{\alpha
_i}>x, \prod_{i=1}^\infty X_i^{\beta_i}>x, M Z_1(\mathbf{X}_3^\infty
)\ge x^{\kappa_1} )}{P(X_1>x^{\kappa_1})P(X_2>x^{\kappa_2})}
\\
\nonumber&&\qquad{} + \frac{P (\prod_{i=1}^\infty X_i^{\alpha_i}>x, \prod_{i=1}^\infty
X_i^{\beta_i}>x, MZ_2(\mathbf{X}_3^\infty)\ge
x^{\kappa_2} )}{P(X_1>x^{\kappa_1})P(X_2>x^{\kappa_2})}
\end{eqnarray}
tends to 0. We will only examine the first term on the right-hand side,
as the second term can be treated analogously.

The numerator of the first term can be bounded by
%
\begin{eqnarray}
\nonumber
&& P \Biggl(\prod_{i=1}^\infty
X_i^{\beta_i}>x, M^{\alpha_1\beta_2-\alpha_2\beta_1} \prod
_{i=3}^\infty X_i^{\alpha
_i\beta_2-\beta_i\alpha_2}\ge
x^{\beta_2-\alpha_2} \Biggr)
\\
&&\quad\leq P \Biggl(\prod_{i=1}^\infty
X_i^{\beta_i}>x, M^{\alpha_1\beta
_2-\alpha_2\beta_1} \prod
_{i=1}^\infty X_i^{\beta_i\alpha_2}\prod
_{i=3}^\infty X_i^{\alpha_i\beta_2-\beta_i\alpha_2}
\ge x^{\beta
_2} \Biggr) \label{smallerthankappa}
\nonumber
\\[-8pt]
\\[-8pt]
\nonumber
&&\quad\leq P \Biggl(\prod_{i=1}^\infty
X_i^{\beta_i}>\min\bigl\{ 1,M^{\vfrac{\alpha_2\beta_1-\alpha_1\beta
_2}{\beta_2}}\bigr\}x,\\
&&\qquad
X_1^{\alpha_2\beta_1/\beta_2}\prod_{i=2}^\infty
X_i^{\alpha_i}\ge\min\bigl\{1,M^{\vfrac{\alpha_2\beta_1-\alpha_1\beta
_2}{\beta_2}}\bigr\} x
\Biggr).\nonumber
\end{eqnarray}
By assertion (i), the asymptotic behavior of the above probability is
related to the linear program
%
\begin{eqnarray}
\label{LPsmalleralpha}
&&\sum_{i=1}^\infty
\tilde{\kappa}_i \to\min!
\nonumber\\[-8pt]\\[-8pt]\nonumber
&&\quad \mbox{under the constraints } (\alpha_2\beta_1/
\beta_2)\tilde{\kappa}_1+\sum
_{i=2}^\infty\alpha_i \tilde{
\kappa_i}\geq1, \sum_{i=1}^\infty
\beta_i\tilde{\kappa}_i\geq1, \tilde{
\kappa}_i \geq0, \forall i \in\mathbb{N}.
\end{eqnarray}
Since $\alpha_2\beta_1/\beta_2<\alpha_1$, all optimal solutions
$(\tilde\kappa_i)_{i \in\mathbb{N}}$ to (\ref{LPsmalleralpha}) are
also feasible solutions to (\ref{LOb}), but they cannot be optimal
solutions to (\ref{LOb}) because the unique optimal solution $(\kappa
_i)_{i \in\mathbb{N}}$ is not a feasible solution of (\ref
{LPsmalleralpha}). Hence, $\sum_{i=1}^\infty\tilde{\kappa}_i>\kappa
_1+\kappa_2+2\delta$ for sufficiently small $\delta>0$ and, by
assertion (i), the right-hand side of (\ref{smallerthankappa}) is of
smaller order than $x^{-(\kappa_1+\kappa_2+\delta)}$, while the
denominator of the first term of the right-hand side of (\ref
{eqremainder}) is of larger order than $x^{-(\kappa_1+\kappa
_2+\delta)}$. Thus, the first summand in (\ref{eqremainder}) tends
to 0. The second summand can be treated analogously, which concludes
the proof in the case $s_0=s_1=1$.

Finally, we prove the assertion (ii) for general $s_0,s_1>0$.
Note that the above value of $D$ in~(\ref{finalconstant}) does not
depend on the distribution of $X_1$ or $X_2$.
Therefore, we may replace $X_1$ and $X_2$ with $z_1X_1$ and $z_2X_2$,
respectively, where
%
\begin{eqnarray}
\label{xfroms} z_1&:=&s_1^{\alpha_2/(\alpha_1\beta_2-\alpha_2\beta
_1)}s_0^{-\beta
_2/(\alpha_1\beta_2-\alpha_2\beta_1)},
\nonumber
\\[-8pt]
\\[-8pt]
\nonumber
z_2&:=&s_0^{\beta
_1/(\alpha_1\beta_2-\alpha_2\beta_1)}s_1^{-\alpha_1/(\alpha
_1\beta_2-\alpha_2\beta_1)},
\end{eqnarray}
such that $z_1^{\alpha_1}z_2^{\alpha_2}=s_0^{-1}$ and $z_1^{\beta
_1}z_2^{\beta_2}=s_1^{-1}$. As this substitution does not alter the
solution to the linear program (\ref{LOb}),
\begin{eqnarray*}
&& \frac{P(Y_0>s_0x, Y_1>s_1x)}{P(X_1>x^{\kappa
_1})P(X_2>x^{\kappa_2})}
\\
&&\quad = \frac{P ((z_1X_1)^{\alpha_1}(z_2X_2)^{\alpha_2}
\prod_{i=3}^\infty X_i^{\alpha_i}>x,(z_1X_1)^{\beta
_1}(z_2X_2)^{\beta_2}\prod_{i=3}^\infty X_i^{\beta_i}>x
)}{P(z_1X_1>x^{\kappa_1})P(z_2X_2>x^{\kappa_2})}
\\
&&\qquad{} \times\frac{P(z_1X_1>x^{\kappa_1})P(z_2X_2>x^{\kappa
_2})}{P(X_1>x^{\kappa_1})P(X_2>x^{\kappa_2})}
\end{eqnarray*}
converges to
\[
D z_1z_2 = D s_0^{(\beta_1-\beta_2)/(\alpha_1\beta_2-\alpha_2\beta_1)}
s_1^{(\alpha_2-\alpha_1)/(\alpha_1\beta_2-\alpha_2\beta_1)}
\]
as $x\to\infty$ by the first part of the proof and the regular
variation of $X_1$ and $X_2$.

\textbf{6. Proof of (iii)}

In the following, we assume that a unique solution $(\kappa_i)_{i \in
\mathbb{N}}$ to (\ref{LOb}) exists with $\kappa_1>0$ and $\kappa
_j=0$ for all $j\geq2$. Furthermore, assume w.l.o.g. that $\alpha
_1\geq\beta_1>0$, thus $\kappa_1=1/\beta_1$. In this case,
%
\begin{eqnarray}\label{degenerate-min}
\nonumber
&& P(Y_0>s_0x, Y_1>s_1x)
\\
\nonumber
&&\quad =P \Biggl(X_1^{\alpha_1}\prod
_{i=2}^\infty X_i^{\alpha
_i}>s_0x,
X_1^{\beta_1}\prod_{i=2}^\infty
X_i^{\beta_i}>s_1x \Biggr)
\nonumber\\[-8pt]\\[-8pt]\nonumber
&&\quad =P \Biggl(X_1^{\beta_1} \Biggl(s_0^{-1}
\prod_{i=2}^\infty X_i^{\alpha_i}
\Biggr)^{\beta_1/\alpha_1}x^{1-\beta_1/\alpha_1}>x, X_1^{\beta_1}s_1^{-1}
\prod_{i=2}^\infty X_i^{\beta_i}>x
\Biggr)
\\
\nonumber &&\quad =P \Biggl(X_1^{\beta_1}\min\Biggl\{
\Biggl(s_0^{-1}\prod_{i=2}^\infty
X_i^{\alpha_i} \Biggr)^{\beta_1/\alpha
_1}x^{1-\beta_1/\alpha_1},
s_1^{-1}\prod_{i=2}^\infty
X_i^{\beta
_i} \Biggr\}>x \Biggr).
\end{eqnarray}
Now, if $\alpha_1=\beta_1$, then the second factor in the product in
(\ref{degenerate-min}) equals $Z:=\min\{
s_0^{-1}\prod_{i=2}^\infty X_i^{\alpha_i}, s_1^{-1} \prod_{i=2}^\infty
X_i^{\beta_i} \}$. Note that
\[
\sum_{i=2}^\infty\kappa_i'>
\kappa_1=\frac{1}{\beta_1}\qquad\mbox{for all } \kappa_i'
\geq0, i \geq2,\qquad \mbox{such that } \sum_{i=2}^\infty
\alpha_i\kappa_i'\geq1,\qquad \sum
_{i=2}^\infty\beta_i\kappa_i'
\geq1 %
\]
since we have assumed a unique solution to (\ref{LOb}). Part (i) of
the statement (applied to $\tilde Y_0:=\prod_{i=2}^\infty X_i^{\alpha
_i}$ and $\tilde Y_1:=\prod_{i=2}^\infty X_i^{\beta_i}$) thus yields
that $P(Z>x)=\mathrm{o}(x^{-1/\beta_1-\varepsilon})$ for some $\varepsilon>0$,
which by Breiman's lemma in turn implies the first case in (\ref
{mproductsdegenerate}).

If $\alpha_1>\beta_1$, then $\beta_i<\beta_1$ for all $i\geq2$,
since otherwise $\kappa_1=1/\beta_1, \kappa_i=0, i\geq2$, would not
be a unique solution to (\ref{LOb}). (If $\beta_i\ge\beta_1$, then
$\kappa_1'=(1-\eps)/\beta_1$, $\kappa_i'=\eps/\beta_i$, for
sufficiently small~$\eps$, and $\kappa_j'=0$ for all $j\in\mathbb
{N} \setminus\{1,i\}$ would satisfy (\ref{LOb}) with $\kappa
_1'+\kappa_i'\le\kappa_1$.) For all $C>0$, the second factor in the
product in (\ref{degenerate-min}) is eventually (for sufficiently
large $x$) bounded from below by $\min\{C \prod_{i=2}^\infty
X_i^{\alpha_i\beta_1/\alpha_1}, s_1^{-1} \prod_{i=2}^\infty
X_i^{\beta_i} \}$. Thus,
\begin{eqnarray*}
&& P \Biggl(X_1^{\beta_1}\min\Biggl\{C\prod
_{i=2}^\infty X_i^{\alpha_i\beta_1/\alpha_1},
s_1^{-1}\prod_{i=2}^\infty
X_i^{\beta_i} \Biggr\}>x \Biggr)
\\
&&\quad \leq P(Y_0>s_0x,Y_1>s_1x)
\leq P \Biggl(X_1^{\beta_1}\mathds{1}_{\{
\prod_{i=2}^\infty X_i^{\alpha_i\beta_1/\alpha_1}>0\}}s_1^{-1}
\prod_{i=2}^\infty X_i^{\beta_i}>x
\Biggr)
\end{eqnarray*}
for $x$ large enough.
For both the left-hand side and the right-hand side of this inequality,
the random variable $X_1^{\beta_1}$ is multiplied with an independent
random variable bounded above by $s_1^{-1}\prod_{i=2}^\infty
X_i^{\beta_i}$. Because $\beta_i<\beta_1$ for all $i \geq2$ we have
$E ( (\prod_{i=2}^\infty X_i^{\beta_i} )^{(1+\varepsilon
)/\beta_1} )<\infty$ for sufficiently small $\varepsilon>0$ by
(\ref{Eqhelpfulbound}). We may thus again apply Breiman's lemma to get
\begin{eqnarray*}
&& E \Biggl( \Biggl(\min\Biggl\{C\prod_{i=2}^\infty
X_i^{\alpha
_i\beta_1/\alpha_1},s_1^{-1} \prod
_{i=2}^\infty X_i^{\beta_i} \Biggr\}
\Biggr)^{1/\beta_1} \Biggr)
\\
&&\quad \leq \liminf_{x \to\infty} \frac
{P(Y_1>s_0x,Y_1>s_1x)}{P(X_1^{\beta_1}>x)}\leq\limsup
_{x \to\infty
} \frac{P(Y_1>s_0x,Y_1>s_1x)}{P(X_1^{\beta_1}>x)}
\\
&&\quad\leq E \Biggl( \Biggl(\mathds{1}_{\{\prod_{i=2}^\infty X_i^{\alpha
_i\beta_1/\alpha_1}>0\}}s_1^{-1}
\prod_{i=2}^\infty X_i^{\beta
_i}
\Biggr)^{1/\beta_1} \Biggr)<\infty.
\end{eqnarray*}
For $C \to\infty$, the lower bound converges to the upper bound,
which yields the second case in (\ref{mproductsdegenerate}). Analogous
arguments with the role of $\alpha_1$ and $\beta_1$ interchanged
conclude the proof of part~(iii) of the statement.
\end{pf*}

\begin{pf*}{Proof of Theorem \ref{power-products-high-dimension}}
By assumption, $X_j\ge\delta$ for some $\delta\in(0,1]$ and all
$1\le j\le n$, and thus $\alpha_{ij}\log X_j\le\alpha_{ij}(\log
X_j)^+-\alpha_{ij}^-\log\delta$. Let $c:=\max_{1\le i\le d}\sum_{j=1}^n
\alpha_{ij}^-\llvert \log\delta\rrvert $. Then, similarly as in the proof
of part (i) of Theorem~\ref{power-products-infinite}, the optimality
of $(\kappa_j)_{1\le j\le n}$ implies
\begin{eqnarray*}
P \Biggl(\prod_{j=1}^n
X_j^{\alpha_{ij}}>y, \forall1\le i\le d \Biggr)& \le& P \Biggl(
\sum_{j=1}^n \alpha_{ij}
\frac{(\log
X_j)^+}{\log y-c}\ge1, \forall1\le i\le d \Biggr)
\\
&=& P \Biggl(\prod_{j=1}^n
\max(X_j,1)\ge\bigl(y{\rm{e}}^{-c}\bigr)^{\sum
_{j=1}^n\kappa_j}
\Biggr)
\\
&=& \mathrm{o} \bigl(y^{\eps-\sum_{j=1}^n \kappa_j} \bigr)
\end{eqnarray*}
for all $\eps>0$ and $y>e^c$. The lower bound can be established in a similar way
as in the proof of Theorem~\ref{power-products-infinite}.

To\vspace*{1pt} prove the second assertion, first note that if a unique optimal
solution of the given form exists, the equation $\mathbf{A}\mathbf
{x}=\mathbf{1}$ must have $(\kappa_j)_{j\in J}=\mathbf{A}^{-1}
\mathbf{1}$ as the unique solution. Let $\tilde{X}_j:=X_j\prod_{i=1}^d
s_i^{-A_{ji}^{-1}}$ for $j \in J$, $Z_i:=\prod_{k\notin
J}X_k^{\alpha_{ik}}$ for $1\le i\le d$ and $\mathbb{M}:= \{
\mathbf{z}\in[0,\infty)^d\mid\prod_{i=1}^d z_i^{A_{ji}^{-1}} \le M
x^{\kappa_j} ~\forall j\in J \}$ for a sufficiently large $M$ (to
be specified later). Then
%
\begin{eqnarray}\label{eqhighdimsplit}
&& P(Y_i>s_i x ~\forall1\le i\le d)
\nonumber
\\
&&\quad  =  \int_{\mathbb{M}} P \biggl( \prod
_{j\in J} \tilde{X}_j^{\alpha
_{ij}}>x/z_i ~\forall1\le i\le d \biggr) P^{(Z_i)_{1\le i\le
d}}(\mathrm{d}\mathbf{z})
\\
\nonumber &&\qquad{} + P \bigl(Y_i>s_ix ~\forall1\le i\le d,
(Z_i)_{1\le i\le d}\notin\mathbb{M} \bigr).
\end{eqnarray}
To show that the last summand is negligible, it suffices to prove that
%
\begin{equation}
\label{eqhighdimnegligible} P \Biggl(Y_i>s_ix ~\forall1\le i\le d,
\prod_{i=1}^d Z_i^{A_{li}^{-1}}
> M x^{\kappa_l} \Biggr)= \mathrm{o}\bigl(x^{-\sum_{j=1}^n \kappa
_j-\delta}\bigr)
\end{equation}
as $x \to\infty$ for some $\delta>0$ and all $l \in J$.
Now,
\begin{eqnarray*}
&& P \Biggl(Y_i>s_ix ~\forall1\le i\le d, \prod
_{i=1}^d Z_i^{A_{li}^{-1}} >
M x^{\kappa_l} \Biggr)
\\
&&\quad  =  P \Biggl(\prod_{j=1}^n
X_j^{\alpha_{ij}}>s_ix ~\forall1\le i\le d, \prod
_{k\notin J} X_k^{\sum_{i=1}^d \alpha_{ik}
A_{li}^{-1}/\kappa_l}>M^{1/\kappa_l}
x \Biggr).
\end{eqnarray*}
Obviously, a feasible solution to the corresponding linear program
%
\begin{eqnarray}
\label{Eq45LP}
&& \sum_{j=1}^n\rho_j \to\min!\nonumber
\\
&&\quad\mbox{under the constraints }
\sum_{j=1}^n \alpha_{ij}
\rho_j \ge1, \forall1\le i \le d,
\\
&&\phantom{\quad\mbox{under the constraints }} \sum
_{k\notin J} \sum_{i=1}^d
\frac{\alpha_{ik} A_{li}^{-1}}{\kappa
_l} \rho_k\ge1, \rho_j\ge0,
\forall1\le j\le n,\nonumber
\end{eqnarray}
%
is also a feasible solution of the original linear program. However, as
the unique solution $(\kappa_j)_{1\le j\le n}$ to the original linear
program does not satisfy the second constraint in (\ref{Eq45LP}), an
optimal solution $(\rho_j)_{1\le j\le n}$ to the above linear program
fulfills $\sum_{j=1}^n \rho_j>\sum_{j=1}^n \kappa_j$. Therefore,
the first assertion (applied to the products $Y_1,\ldots,Y_d$ and
$\prod_{i=1}^d Z_i^{A_{li}^{-1}/\kappa_l}$) implies (\ref
{eqhighdimnegligible}).

The distribution of the random vector $(\tilde{X}_j)_{j\in J}$ is
regularly varying on $\mathbb{E}^d$ in the sense of Lindskog, Resnick
and Roy \cite{LiReRo13}, Definition 3.2, w.r.t. the
``multiplication'' $(x, (t_j)_{j\in J})\mapsto(x^{\kappa_j}t_j)_{j\in
J}$, because
%
\begin{equation}
\label{eqgenregvar} \frac{P((\tilde{X}_j)_{j\in J}\in\bigtimes_{j\in J}
(x^{\kappa
_j}t_j,\infty))}{P((\tilde{X}_j)_{j\in J}\in\bigtimes_{j\in J}
(x^{\kappa_j},\infty))} \to\prod_{j\in J}t_j^{-1}=:
\nu\biggl(\bigtimesd_{j\in J} (t_j,\infty) \biggr)
\end{equation}
for all $t_j>0$, $j\in J$
(cf. Example 3.1 of that paper).

By our assumptions, $\mathbf{A}^{-1}(\mathbb{E}^d)\subset\mathbb
{E}^d$, so that
%
\begin{equation}
\label{Eqsubset} \exp\Biggl(\mathbf{A}^{-1} \Biggl(\bigtimesd_{i=1}^d
(-\log z_i,\infty) \Biggr) \Biggr) \subset\exp\bigl(\bigl(-
\mathbf{A}^{-1}(\log z_i)_{1 \leq i \leq d},\bolds{\infty
}\bigr) \bigr)
\end{equation}
which is bounded away from the boundary of $\mathbb{E}^d$ for fixed
$(z_i)_{1 \leq i \leq d}$.
Hence, for the integrand of the first summand in (\ref
{eqhighdimsplit}) we obtain
%
\begin{eqnarray}
\label{Eqhigdimint}
&& \frac{P ( \prod_{j\in J} \tilde
{X}_j^{\alpha_{ij}}>x/z_i ~\forall 1\le i\le d )}{\prod_{j\in
J}P(\tilde{X}_j>x^{\kappa_j})}\nonumber
\\
\nonumber
&&\quad  =  \frac{P ( \mathbf{A}(\log\tilde{X}_j)_{j\in
J}\in\bigtimes_{i=1}^d (\log(x/z_i),\infty) )}{P ((\tilde
{X}_j)_{j\in J}\in\bigtimes_{j \in J}(x^{\kappa_j},\infty) )}
\nonumber\\[-8pt]\\[-8pt]\nonumber
&&\quad  =  \frac{P ( (\tilde{X}_j)_{j\in J}\in(x^{\kappa
_j})_{j\in J}\exp(\mathbf{A}^{-1} (\bigtimes_{i=1}^d (-\log
z_i,\infty) ) ) )}{P ((\tilde{X}_j)_{j\in J}\in
\bigtimes_{j \in J}(x^{\kappa_j},\infty) )}
\\
\nonumber
&&\quad  \to \nu\Biggl(\exp\Biggl(\mathbf{A}^{-1} \Biggl(\bigtimesd
_{i=1}^d (-\log z_i,\infty) \Biggr) \Biggr)
\Biggr),
\end{eqnarray}
where by convention $(x^{\kappa_j})_{j\in J}B=\{(x^{\kappa_j}y_j)_{j
\in J} \mid(y_j)_{j \in J} \in B\}$ for $B \subset\mathbb{E}^d$ and the
measure $\nu$ is defined by (\ref{eqgenregvar}). If we show that
there exists an integrable majorant to (\ref{Eqhigdimint}), then by
(\ref{eqhighdimsplit}), (\ref{eqhighdimnegligible}) and dominated
convergence
\[
\frac{P(Y_i>s_ix ~\forall 1\le i\le d)}{\prod_{j\in J}P(\tilde
{X}_j>x^{\kappa_j})} \to\int\nu\Biggl(\exp\Biggl(\mathbf{A}^{-1}
\Biggl(\bigtimesd_{i=1}^d (-\log z_i,\infty)
\Biggr) \Biggr) \Biggr) P^{(Z_i)_{1\le i\le d}}(\mathrm{d}\mathbf{z}). %
\]
Direct calculations show that this limit equals
\[
\frac{1}{\llvert \det\mathbf{A}\rrvert }E \Biggl(\prod_{i=1}^d
\frac{Z_i^{\sum
_{j \in J}A_{ji}^{-1}}}{\sum_{j \in J}A_{ji}^{-1}} \Biggr)=D %
\]
and thus (\ref{asymphighdimension}) follows with
\[
\lim_{x \to\infty} \frac{\prod_{j\in J}P(\tilde{X}_j>x^{\kappa
_j})}{\prod_{j\in J}P(X_j>x^{\kappa_j})}=\prod
_{i=1}^ds_i^{-\sum_{j
\in J}A_{ji}^{-1}}.
\]
It remains to define an integrable majorant to (\ref{Eqhigdimint}).
As shown above [cf. (\ref{Eqsubset})],
\begin{eqnarray*}
&& \frac{P ( \prod_{j\in J} \tilde{X}_j^{\alpha
_{ij}}>x/z_i ~\forall 1\le i\le d )}{\prod_{j\in J}P(\tilde
{X}_j>x^{\kappa_j})}
\\
&&\quad = \frac{P ( (\tilde{X}_j)_{j\in J}\in\exp(\mathbf
{A}^{-1} (\bigtimes_{i=1}^d (\log(x/z_i),\infty) )
) )}{\prod_{j\in J}P(\tilde{X}_j>x^{\kappa_j})}
\\
&&\quad \leq \frac{P ( (\tilde{X}_j)_{j\in J}\in(x^{\kappa_j})_{j
\in J}\exp( (\mathbf{A}^{-1}(-\log z_i)_{1 \leq i \leq
d},\bolds{\infty} ) ) )}{\prod_{j\in
J}P(\tilde{X}_j>x^{\kappa_j})}.
\end{eqnarray*}
Hence, for $\varepsilon>0$ the Potter bounds yield
\begin{eqnarray*}
\frac{P ( \prod_{j\in J} \tilde{X}_j^{\alpha_{ij}}>x/z_i~\forall 1\le i\le d )}{\prod_{j\in J}P(\tilde{X}_j>x^{\kappa_j})} &\le
& \prod_{j\in J}
\frac{P(\tilde{X}_j>x^{\kappa_j} \prod_{i=1}^d z_i^{-
A_{ji}^{-1}})}{P(\tilde{X}_j>x^{\kappa_j})}
\\
& \le& (1+\varepsilon) \Biggl(\prod_{j\in J}\prod
_{i=1}^d z_i^{A_{ji}^{-1}}
\Biggr)^{1\pm\eps}
\end{eqnarray*}
for $\mathbf{z}\in\mathbb{M}$ if $M$ is chosen sufficiently large. Since
\[
\prod_{j\in J}\prod_{i=1}^d
Z_i^{A_{ji}^{-1}} = \prod_{k\notin J}
X_k^{\sum_{j\in J} \sum_{i=1}^d \alpha_{ik} A^{-1}_{ji}}, %
\]
and the $\tilde{X}_k$ are independent and regularly varying with index
$-1$, it suffices to show that $\sum_{j\in J} \sum_{i=1}^d \alpha
_{ik} A^{-1}_{ji}<1$ for all $k\notin J$. As in the proof of Theorem
\ref{power-products-infinite}, the strong complementary slackness
theorem implies that the corresponding dual problem
\begin{eqnarray*}
&& \sum_{i=1}^d\hat
\kappa_i \to\max!
\\
&&\quad \mbox{under the constraints }\sum_{i=1}^d
\alpha_{ij}\hat\kappa_i\le1, \forall1\le j\le n,
\hat\kappa_i\ge0, \forall1\le i\le d,
\end{eqnarray*}
has a solution that satisfies
$\sum_{i=1}^d \alpha_{ij}\hat\kappa_i=1$ for all $j\in J$, that is,
$\hat\kappa_i=\sum_{j\in J}A_{ji}^{-1}$ for all $1\le i\le d$, and
$\sum_{j\in J} \sum_{i=1}^d \alpha_{ik}A^{-1}_{ji}=\sum_{i=1}^d
\alpha_{ik}\hat\kappa_i<1$ for all $k\notin J$, which concludes the proof.
\end{pf*}
%
\subsection{Proofs to Section~\texorpdfstring{\protect\ref{sectasymptoticsSV}}{5}}

\begin{pf*}{Proof of Theorem \ref{limitmeasuressvmodels}}
Note that if a feasible solution $(\kappa_i)_{i\in\N_0}$ of (\ref
{svmin}) satisfies $\kappa_i>0$ for exactly one index $i\in\N_0$ and
$\alpha_i>\alpha_{i-h}$, then $\tilde\kappa_i:=1/\alpha_i$,
$\tilde\kappa_{i+h}:=(\alpha_i-\alpha_{i-h})/\alpha_i^2$ and
$\tilde\kappa_j:=0$ for all other $j\in\N_0$ defines a feasible
solution, too, with $\tilde\kappa_i+\tilde\kappa_{i+h}<1/\alpha
_{i-h}\le\kappa_i$. Hence, the first feasible solution cannot be
optimal, which (together with an analogous argument in the case $\alpha
_i<\alpha_{i-h}$) proves that for an optimal solution necessarily
$\alpha_i=\alpha_{i-h}$ (and thus $i\ge h$), if $i\in\N_0$ is the
only index with $\kappa_i>0$.

Define coefficients $\hat\alpha_i:=\alpha_{i-h}$ and $\hat\beta
_i:=\alpha_i$ (with the convention $\alpha_k=0$ for $k<0$) and let
$\hat X_i:= {\rm{e}}^{\eta_{h-i}}$, $i\in\N_0$, so that $(\sigma
_0, \sigma_h)=(\prod_{i=0}^\infty\hat X_i^{\hat\alpha_i}, \prod
_{i=0}^\infty\hat X_i^{\hat\beta_i})$.
The random variables $\hat X_i$ are regularly varying with index $-1$,
and the coefficients $\hat\alpha_i$ and $\hat\beta_i$ satisfy the
conditions of Theorem~\ref{power-products-infinite} (where we start
from the index $i=0$ instead of $i=1$). Moreover, by Remark~\ref
{remconvconds}(ii) and the assumption that $E(\eta_0^2)<\infty$, it
follows that $\sigma_0, \sigma_h>0$ almost surely.
The statement about $(\sigma_0, \sigma_h)$ thus follows by an
application of Theorem~\ref{power-products-infinite}, cases (ii) and
(iii), in combination with Remark~\ref{remHRV}(i)--(iii).

In particular, $(\sigma_0,\sigma_h)$ is regularly varying on $\mathbb
{E}^2$ with index $-\sum_{i=1}^\infty\kappa_i$. Hence, $\eps_0$
fulfills the moment condition of Corollary~\ref{corSVcor}, which
yields the assertion on $(X_0, X_h)$.
\end{pf*}

\begin{pf*}{Proof of Corollary \ref{limitmeasuressvmodelsmonotone}}
The stated unique form of the solution is immediate from the assumed
strict monotonicity of the coefficients. The form of the limit measure
then follows from Theorem~\ref{limitmeasuressvmodels}(i) with $i=0,
j=h$ and $\alpha_0=1$.
\end{pf*}

\begin{pf*}{Proof of Theorem \ref{representeta}}
Let $\alpha_{2m(i-1)}=1$ and $\alpha_{2m(i-1)+i}=2-\eta_i^{-1} \in
[0,1]$ for $1\leq i \leq m$ and $\alpha_j=0$ else in Definition~\ref
{defSVnew}. This choice guarantees that for each $1 \leq h \leq m$
with $\eta_h>1/2$ there exists exactly one $k(h) \in\mathbb{N}_0$
such that both $\alpha_{k(h)}$ and $\alpha_{k(h)+h}$ are positive;
furthermore, for this $k(h)$ one has $\alpha_{k(h)}=1$ and $\alpha
_{k(h)+h}=2-\eta_h^{-1}$.

Fix $1\leq h \leq m$. If $\eta_h=1/2$, then there exists no $i$ such
that both $\alpha_i$ and $\alpha_{i+h}$ are positive and thus $\sigma
_0$ and $\sigma_h$ are independent and $(\sigma_0,\sigma_h)$ has a
coefficient of tail dependence equal to $1/2$.

If $\eta_h=1$, then $\alpha_{k}=\alpha_{k+h}=1$ for exactly one
$k=k(h) \in\mathbb{N}_0$ and thus $\kappa_{k(h)+h}=1$, $\kappa_j=0, j
\neq k(h)+h$, is the unique solution to the optimization problem (\ref
{svmin}). By Theorem~\ref{limitmeasuressvmodels}(ii), the coefficient
of tail dependence of $(\sigma_0,\sigma_h)$ thus equals $1$.

Finally, if $\eta_h \in(1/2,1)$, write
%
\begin{equation}
\label{foruniquesolution}\sigma_0=\prod_{i \in
\mathbb{Z}\dvtx \alpha_i >0} {
\rm{e}}^{\alpha_i\eta_{-i}}, \qquad\sigma_h=Z_h \cdot\prod
_{i\in\mathbb{Z}\dvtx \alpha_{i+h} \in(0,1)} {\rm{e}}^{\alpha_{i+h}\eta_{-i}}
\end{equation}
with\vspace*{1pt} $\alpha_j:=0$ for $j<0$, and $Z_h:=\prod_{i \in\mathbb
{Z}\dvtx \alpha_{i+h}=1} {\rm{e}}^{\eta_{-i}}$, which is independent of
all other factors on the right-hand sides of (\ref
{foruniquesolution}), as $\alpha_{i+h}=1$ implies $\alpha_i=0$.
Moreover, according to the corollary to Theorem 3 of Embrechts and
Goldie \cite{EmGo80}, $Z_h$ is regularly varying with index $-1$.
Now, the joint behavior of $\sigma_0$ and $\sigma_h$ can be derived
by applying Theorem~\ref{power-products-infinite} to the
representation~(\ref{foruniquesolution}).
Observe that the unique optimal solution to the corresponding
optimization problem to minimize $\tilde\kappa+\sum_{i\in\Z}
\kappa_i$ under the constraints $\sum_{i\in\Z} \kappa_i\alpha
_i\ge1, \tilde\kappa+\sum_{i\in\Z\dvtx  \alpha_{i+h}\in(0,1)} \kappa
_i\alpha_{i+h}\ge1$ and $\tilde{\kappa} \geq0, \kappa_i \geq0, i
\in\mathbb{Z}$, is given by $\kappa_{k(h)}=1$, $\tilde\kappa
=1-\alpha_{k(h)+h}=\eta_h^{-1}-1$ and $\kappa_j=0$ for all other
$j\in\Z$, because $\kappa_{k(h)}$ is the only value which
contributes to both sums of the constraints and, at the same time, it
is multiplied with the largest coefficient in the first sum. Therefore,
according to Theorem~\ref{power-products-infinite},
\[
P(\sigma_0>x, \sigma_h>x) \sim cP\bigl({
\rm{e}}^{\eta
_{-k(h)}}>x\bigr)P\bigl(Z_h>x^{\eta_h^{-1}-1}\bigr)
\]
for some constant $c>0$, and the coefficient of tail dependence of
$(\sigma_0, \sigma_h)$ equals $1/(\kappa_{k(h)}+\tilde\kappa)=\eta_h$.

By Corollary~\ref{corSVcor} and the following note, the same holds
true for $(X_0,X_h)$ (and $(\llvert X_0\rrvert, \llvert X_h\rrvert
)$) for all $\eta_h \in
[1/2,1]$ if $E(\llvert \eps_0\rrvert ^{2+\delta})<\infty$ for some
$\delta>0$.
\end{pf*}

\textbf{\textit{Details for Remark}~\ref{asymptoticdependence}(ii).} If $\eta_h=1$ for some $h>0$, then any
optimal solution to (\ref{svmin}) has to satisfy $\sum_{i=0}^\infty
\kappa_i=1$. (This even holds in the case that the solution is not
unique since otherwise Theorem~\ref{power-products-infinite}(i) leads
to a contradiction to $\eta_h=1$.) But then $\kappa_i>0$ can only
hold if $\alpha_{i-h}=\alpha_i=1$. Write
\[
\sigma_0=Z \prod_{i\geq h\dvtx  \min\{\alpha_i, \alpha_{i-h}\}<1}
\mathrm{e}^{\alpha_{i-h}\eta_{h-i} }, \qquad\sigma_h=Z \prod
_{i
\in\mathbb{N}_0\dvtx  \min\{\alpha_i, \alpha_{i-h}\}<1} \mathrm{e}^{\alpha
_i\eta_{h-i} }
\]
with\vspace*{1pt} $Z:=\prod_{i \in\mathbb{N}_0\dvtx  \alpha_i=\alpha_{i-h}=1}\mathrm
{e}^{\eta_{h-i}}$. Again, $Z$ is regularly varying with index $-1$ by
Embrechts and Goldie \cite{EmGo80}. Then the corresponding linear program
\[
\tilde{\kappa}+\sum_{i\in\mathbb{N}_0\dvtx  \min\{\alpha_i, \alpha
_{i-h}\}<1}\kappa_i \to
\min!
\]
under the constraints
\begin{eqnarray}
\tilde{\kappa}+\sum_{i\geq h\dvtx  \min\{\alpha_i, \alpha_{i-h}\}
<1}\alpha_{i-h}
\kappa_i\geq1,\qquad \tilde{\kappa}+\sum_{i\in
\mathbb{N}_0\dvtx  \min\{\alpha_i, \alpha_{i-h}\}<1}
\alpha_i\kappa_i\geq1,\nonumber
\\
\eqntext{\tilde{\kappa}\geq0, \kappa_i \geq0, \forall i \in
\mathbb{N}_0,} %
\end{eqnarray}
has the unique solution $\tilde{\kappa}=1, \kappa_i=0$ for all $ i
\in\mathbb{N}_0$, and thus $P(\sigma_0>x, \sigma_h>x) \sim cP(Z>x)$
for some constant $c>0$ by Theorem~\ref{power-products-infinite}. By
Lemma 7.2 of Rootz{\'e}n \cite{Ro86}
\[
P(Z>x)= P \biggl(\sum_{i \in\mathbb{N}_0\dvtx  \alpha_i=\alpha
_{i+h}=1}\eta_{-i}>
\log x \biggr) \sim\hat{K} (\log x)^\beta x^{-1},
\]
for a constant $\hat{K}>0$, where we have used equation (7.8) in
Rootz{\'e}n \cite{Ro86} and the fact that $\beta<-1$. On
the other hand,
\[
P(\sigma_0>x)= P \Biggl(\sum_{i=0}^\infty
\alpha_i \eta_{-i}>\log x \Biggr) \sim
\hat{K}' (\log x)^\beta x^{-1},
\]
for some constant $\hat{K}'>0$ by the same arguments as above.
Combining the above asymptotics, we arrive at
\[
\lim_{x \to\infty} P(\sigma_h>x \mid\sigma_0>x)=
\lim_{x \to
\infty} \frac{P(\sigma_0>x, \sigma_h>x)}{P(\sigma_0>x)}>0 %
\]
and thus asymptotic dependence of $(\sigma_0, \sigma_h)$. The same
holds true for $(X_0, X_h)$.

In contrast, if $\beta>-1$, then one may conclude from Lemma 7.2 of
Rootz{\'e}n \cite{Ro86} that $P(\sigma_h>x \mid\sigma
_0>x)=\mathrm{O}((\log x)^{-l(\beta+1)})\to0$ as $x\to\infty$ with $l:=\sum
_{i=0}^\infty\mathds{1}_{\{\alpha_i=1\}}- \sum_{i=0}^\infty\mathds
{1}_{\{\alpha_i=\alpha_{i+h}=1\}}>0$, which confirms the asymptotic
independence of $\sigma_0$ and $\sigma_h$ in this case.


\section*{Acknowledgments}
The authors would like to thank the Associate
Editor and two anonymous referees for their careful reading and many
helpful comments. This project was supported by the German Research
Foundation DFG, Grant no JA 2160/1.


%

\printhistory
\end{document}